\documentclass[11pt,a4paper,reqno]{amsart}
\usepackage{amsmath,amscd,amstext,amsthm,amsfonts,latexsym}
\usepackage[dvips]{graphicx}
\usepackage{enumitem}
\usepackage{bbm}
\usepackage[colorlinks=true, linkcolor=blue, urlcolor=red, citecolor=blue, hyperindex, backref]{hyperref}
\usepackage[dvipsnames]{xcolor}
\usepackage{amssymb}

\theoremstyle{plain}
\newtheorem{theorem}{Theorem}[section]
\newtheorem{proposition}[theorem]{Proposition}
\newtheorem{lemma}[theorem]{Lemma}
\newtheorem{corollary}[theorem]{Corollary}
\newtheorem{maintheorem}{Theorem}

\theoremstyle{definition}
\newtheorem{remark}[theorem]{Remark}
\newtheorem{example}[theorem]{Example}
\newtheorem{definition}[theorem]{Definition}

\newtheorem{question}{Question}

\newcommand{\field}[1]{\mathbb{#1}}
\newcommand{\RR}{\field{R}}
\newcommand{\CC}{\field{C}}
\newcommand{\ZZ}{\field{Z}}
\newcommand{\NN}{\field{N}}

\newcommand{\SSS}{\mathbb{S}}
\newcommand{\KK}{\field{K}}

\newcommand{\PP}{\field{P}}

\newcommand{\ga} {\gamma}    \newcommand{\Ga}{\Gamma}
       
\newcommand{\ep} {\varepsilon}

\newcommand{\ro}{\rho}

\newcommand{\si} {\sigma}

\newcommand{\om} {\omega}       

\newcommand{\cU}{{\mathcal U}}

\newcommand{\cP}{{\mathcal P}}

\newcommand{\cD}{{\mathcal D}}

\newcommand{\cE}{{\mathcal E}}

\newcommand{\diam}{\operatorname{{diam}}}

\newcommand{\graph}{\operatorname{{graph}}}

\newcommand{\supp}{\operatorname{{supp}}}

\title{Metric mean dimension via subshifts of compact type}

\begin{document}



\author[Gustavo Pessil]{Gustavo Pessil}
\address{CMUP \& Departamento de Matem\'atica, Faculdade de Ci\^encias da Universidade do Porto, Rua do Campo Alegre 687, Porto, Portugal.}
\email{Gustavo.pessil@outlook.com}




\date{\today}

\begin{abstract}

We investigate the metric mean dimension of subshifts of compact type. We prove that the metric mean dimensions of a continuous map and its inverse limit coincide, generalizing Bowen's entropy formula. Building upon this result, we extend the notion of metric mean dimension to discontinuous maps in terms of suitable subshifts. As an application, we show that the metric mean dimension of the Gauss map and that of induced maps of the Manneville-Pomeau family is equal to the box dimension of the corresponding set of discontinuity points, which also coincides with a critical parameter of the pressure operator associated to the geometric potential.

\end{abstract}

\maketitle

\tiny
\tableofcontents
\normalsize

\section{Introduction}

Symbolic dynamics on finite alphabets are classical mathematical objects that have brought forward a great variety of dynamics and intervened in major achievements in Dynamical Systems and Ergodic Theory. The main invariant in these areas is the entropy, which may be expressed through both a topological and a measure-theoretical perspective. For subshifts on finite alphabets the topological entropy quantifies the exponential growth rate of the number of finite words of fixed length. Furthermore, for a Markov subshift the entropy equals the logarithm of the spectral radius of the generating graph, which may be read as the spectral radius of a linear operator.

For compact alphabets this interpretation of complexity in terms of counting words is no longer feasible. Nevertheless, we may still follow the same guiding principle by using a dimensional approach. More precisely, for a given scale, one identifies all but finitely many letters prior to counting and defines the entropy at that scale; the topological entropy is thus obtained by refining the scale. However, in this setting, its value may be infinite. The metric mean dimension introduced by E. Lindenstrauss and B. Weiss in \cite{LW2000} is a geometric invariant which is useful to distinguish precisely those systems with infinite topological entropy. It measures the speed at which the entropy at a given scale goes to infinity as the scale approaches zero, and it can be seen as a dynamical analogue of the box dimension. For example, the metric mean dimension of a full shift on a compact alphabet, endowed with a suitable metric, is exactly the box dimension of the alphabet.

There is an asymmetry in the definition of topological entropy since it looks only at the future orbits of points, and the same happens with the metric mean dimension. When the map is invertible, it turns out that its inverse has the same entropy. On the contrary, when the map is not invertible, the definition of entropy cannot be reversed in time. Yet, a non-invertible map induces a shift homeomorphism on the corresponding inverse limit space, and Bowen showed in \cite{Bowen} that the entropy of such a shift homeomorphism is equal to the entropy of the original map.

Generalizing the class of subshifts of finite type, Friedland introduced in \cite{Fried} the analogous concept within the compact alphabet setting, which we refer to as subshifts of compact type, and proved that the topological entropy of the unilateral and bilateral subshifts induced by a closed transition set coincide. This relation was already known for subshifts of finite type on finite alphabets, in which case the entropy is equal to the spectral radius of the transition matrix. In the particular case where the transition set is the graph of a continuous map, the associated bilateral subshift is a reformulation of the inverse limit of the map. Altogether, Friedland's results generalize Bowen's formula to arbitrary subshifts of compact type. In this paper we shall establish analogous properties for the metric mean dimension (see Theorems~\ref{n=z} and \ref{corodef}). As a byproduct, we provide an upper bound for the metric mean dimension of an arbitrary map acting on a compact metric space in terms of spectral radii of scale-dependent matrices (see Section~\ref{specradiussection}). 

An application of our results, which is of independent interest, is that they motivate a definition of metric mean dimension for discontinuous maps on compact metric spaces in terms of suitable subshifts of compact type. To illustrate the scope of this new concept, we will show that the metric mean dimension of an interval map with infinitely many full branches coincides with the box dimension of its sets of critical points (see Theorem~\ref{gausslikethm}). In particular, this setting comprises the Gauss map and the Young-induced maps of the Manneville-Pomeau family.

The paper is organized as follows. In the remainder of this section we will briefly recall a few definitions, state our main results and address some applications. In Sections~\ref{prelim} and \ref{specradiussection} we collect some auxiliary material. Sections~\ref{proofA}, \ref{proofB}, \ref{proofcor}, \ref{proofC} are devoted to the proofs.

\subsection{Subshifts of compact type} Let $(X,d)$ be a compact metric space and $\KK=\NN$ or $\ZZ$. In the rest of the paper, we fix some $\ro>1$ and endow the product space $X^\KK$ with the metric \begin{equation}\label{metricshift}
d^\KK(x,y)\,=\,\sup_{i\,\in\,\KK}\frac{d(x_i,y_i)}{\ro^{|i-1|}}.
\end{equation} 
 Unless stated otherwise, we will always consider $\ro=2$. A subset $Y\subset X^\KK$ is called a \emph{subshift} if it is closed and invariant by the shift action  $$\si_\KK\big( (x_i)_{i\in \KK} \big) = (x_{i+1})_{i\in \KK}.$$ We recall that the metric mean dimension of the full shift is given by (see the precise definitions in Section~\ref{prelim})\begin{equation}\label{fullshift}
    \begin{split}    \overline{\mathrm{mdim}}_M(X^\KK,d^\KK,\si_\KK)\,&=\, \overline{\mathrm{dim}}_B(X,d)\\\underline{\mathrm{mdim}}_M(X^\KK,d^\KK,\si_\KK)\,&=\, \underline{\mathrm{dim}}_B(X,d)
    \end{split}
\end{equation}(see \cite[Theorem 5]{VV} and also \cite[Theorem D]{CPV} for a version with potential).

Consider a subset $\Ga\,\subset X\times X$, whose role will be to to prescribe transitions, and the induced set of admissible sequences $$\Ga_\KK\,=\,\{ (x_i)_{i\in \KK}\,\colon (x_i,x_{i+1})\in\Ga,\,\,\forall\, i\in\KK \},$$ 
which is invariant and we always assume to be non-empty. We refer to $\overline{\Ga_\KK}$ as a \emph{subshift of compact type} and note that $\Ga_\KK$ is already closed whenever $\Ga$ is. This model was proposed by Friedland (cf. \cite{Fried}) in order to assign a notion of entropy to finitely generated free semigroup actions. We refer the reader to Section~\ref{proofB} for more information regarding such actions.

It was proved in \cite[Theorem 3.1]{Fried} that the topological entropy of the unilateral and bilateral subshifts induced by a given closed set coincide. Our first result shows that this also holds in the case of the metric mean dimension.

\begin{maintheorem}\label{n=z}
Let $(X,d)$ be a compact metric space and $\Ga\subset X\times X$ be a closed subset. Then, \begin{eqnarray*}     \overline{\mathrm{mdim}}_M(\Ga_\NN,d^\NN,\si_\NN)\,&=&\, \overline{\mathrm{mdim}}_M(\Ga_\ZZ,d^\ZZ,\si_\ZZ) \\ \underline{\mathrm{mdim}}_M(\Ga_\NN,d^\NN,\si_\NN)\,&=&\, \underline{\mathrm{mdim}}_M(\Ga_\ZZ,d^\ZZ,\si_\ZZ).
\end{eqnarray*}
\end{maintheorem}

A particular instance to which we can apply Theorem~\ref{n=z} is when we consider $\KK\,=\,\ZZ$ and $\Ga\,=\,\graph(T)$, where $T\colon X\to X$  is a continuous map. In this case, the subshift $$\Ga_\ZZ\,=\,\{(x_i)_{i\in\ZZ}\colon T(x_i)=x_{i+1}\}$$ is refered to as the \emph{inverse limit of $T$}. In the entropy setting, Bowen's formula is an immediate consequence of Friedland's~ \cite[Theorem 3.1]{Fried}, since $(X,T)$ and $(\Ga_\NN,\si_\NN)$ are topologically conjugate. Even though the metric mean dimension is not a topological invariant (though it is a bi-Lipschitz one), we show that it is the same for these two systems. The next result summarizes this information.

\begin{maintheorem}\label{corodef}
    Let $(X,d)$ be a compact metric space and $T\colon X\to X$ be a continuous map. Then\begin{equation}\label{invlimit}
\begin{split}
\overline{\mathrm{mdim}}_M(X,d,T)&\,=\,\overline{\mathrm{mdim}}_M(\Ga_\NN,d^\NN,\si_\NN)
\,=\,\overline{\mathrm{mdim}}_M(\Ga_\ZZ,d^\ZZ,\si_\ZZ) \\ \underline{\mathrm{mdim}}_M(X,d,T)&\,=\,\underline{\mathrm{mdim}}_M(\Ga_\NN,d^\NN,\si_\NN)
\,=\,\underline{\mathrm{mdim}}_M(\Ga_\ZZ,d^\ZZ,\si_\ZZ). \end{split}
\end{equation}
\end{maintheorem}

Some comments are in order. Clearly, Theorem~\ref{corodef} hints that it is worthwhile investigating the metric mean dimension of subshifts on compact alphabets, a topic that has recently received much attention (see \cite{GS2}, \cite{GT},  \cite{Tsu2}), though far less known than the finite alphabet case. We remark that Theorems~\ref{n=z} and \ref{corodef} were first proved in the author's Master Thesis \cite{mestrado}. After completing this work, we became aware that a version of Theorem~\ref{corodef} for surjective maps was recently published in \cite[Lemma 3.8]{BS}.

Under appropriate adjustments, Theorem~\ref{corodef} allows us to extend properties valid for invertible maps to non-invertible ones. Let us illustrate this feature. Given a continuous map $T$ acting on a compact metric space $(X,d)$, denote by $\cP_T(X)$ the set of $T-$invariant Borel probability measures endowed with the weak$^*$ topology, and by $\cE_T(X)$ its subset of ergodic elements. Consider the \textit{local metric mean dimension function} introduced in \cite{CPV}, defined as
\begin{eqnarray}\label{def:D}
\mathcal{D} \colon X \,\,&\to&\,\,\RR \nonumber \\
x &\mapsto& \inf \,\Big\{\overline{\mathrm{mdim}}_M(\overline{U},d,T)\colon \,U \textrm{ is an open neighbourhood of }x \Big\}.
\end{eqnarray}
The map $\mathcal{D}$ is upper semi-continuous and constant $\mu-$almost everywhere with respect to any ergodic probability measure (cf. \cite[Lemmas 9.1 and 9.2]{CPV}) and has been connected to different notions of measure-theoretic metric mean dimension (cf. \cite[Corollary 9.6 and Example 10.5]{CPV}). The next result was established in \cite[Theorem E]{CPV} under the assumption that the map $T$ is a homeomorphism. Using Theorem~\ref{corodef}, we may drop this condition.

\begin{corollary}\label{localmdim}
Let $(X,d)$ be a compact metric space and $T\colon X \to X$ be a continuous map such that $\overline{\mathrm{mdim}}_M(X,d,T) < +\infty$. Then
\begin{equation}\label{eqlocalmdim}
    \overline{\mathrm{mdim}}_M(X,d,T) \,=\, \max_{\mu\,\in\,\cP_T(X)} \,\int \mathcal{D}(x) \,d\mu(x) \,=\, \max_{\mu\,\in\,\cE_T(X)} \,\int \mathcal{D}(x) \,d\mu(x).
\end{equation}
In addition, a measure $\mu \in \cP_T(X)$ attains the previous maximum if and only if
$$\mathcal{D}|_{\supp(\mu)} \,\equiv \, \overline{\mathrm{mdim}}_M(X,d,T).$$
\end{corollary} 

An element $x\in X$ is said to be a \textit{metric mean dimension point} if $\mathcal{D}(x)>0$; it is a \textit{full metric mean dimension point} if $\mathcal{D}(x)=\overline{\mathrm{mdim}}_M(X,d,T).$ Denote the set of such points by $D_p(X,d,T)$ and $D_p^{full}(X,d,T)$, respectively. The following is an immediate consequence of Corollary~\ref{localmdim}, which, in particular, guarantees the existence of full metric mean dimension points.

\begin{corollary}
Let $(X,d)$ be a compact metric space and $T\colon X\to X$ be a continuous map. Then there exists an ergodic probability measure $\mu \in \cE_T(X)$ such that $\supp(\mu) \subseteq D_p^{full}(X,d,T).$
\end{corollary}

\begin{example}\label{defcaixinhas}
 Let $f\colon [0,1] \to [0,1]$ be given by $f(x) = |1-|3x-1||$ and $(a_k)_{k \, \in \, \mathbb N \,\cup\, \{0\}}$ be the sequence whose general term is $a_0 = 0$ and $a_k \,=\, \sum_{j=1}^k \, \frac{6}{\pi j^2}.$ For each $k \in \mathbb N$, take the interval $J_k \,=\, [a_{k-1}, a_k]$ and let $T_k\colon J_k \to [0,1]$ be the unique increasing affine map from $J_k$ onto $[0,1]$. Let
$$
\begin{array}{rccc}
T \colon & [0,1] & \rightarrow & [0,1] \\
& x \in J_k & \mapsto & T_k^{-1}\circ f^{k-1} \circ T_k\\
& x=1 & \mapsto & 1
\end{array}.
$$ (see Figure~\ref{caixinhas}). It was shown in \cite[Proposition 8]{VV} that $\mathrm{mdim}_M([0,1],d,T)\,=\,1$ and in \cite[Example 10.3]{CPV} that $$\cD(x)\,=\left\{ \begin{array}{cc}
   1  & \text{if $x=1$} \\
     0 & \text{otherwise}
    \end{array} .
   \right.$$ In particular, the unique invariant probability measure maximizing \eqref{eqlocalmdim} is the Dirac measure supported at $1$ and $D_p(X,d,T)\,=\,D_p^{full}(X,d,T)\,=\,\{1\}.$

    \begin{figure}[!htb]
\begin{center}
\includegraphics[scale=0.3]{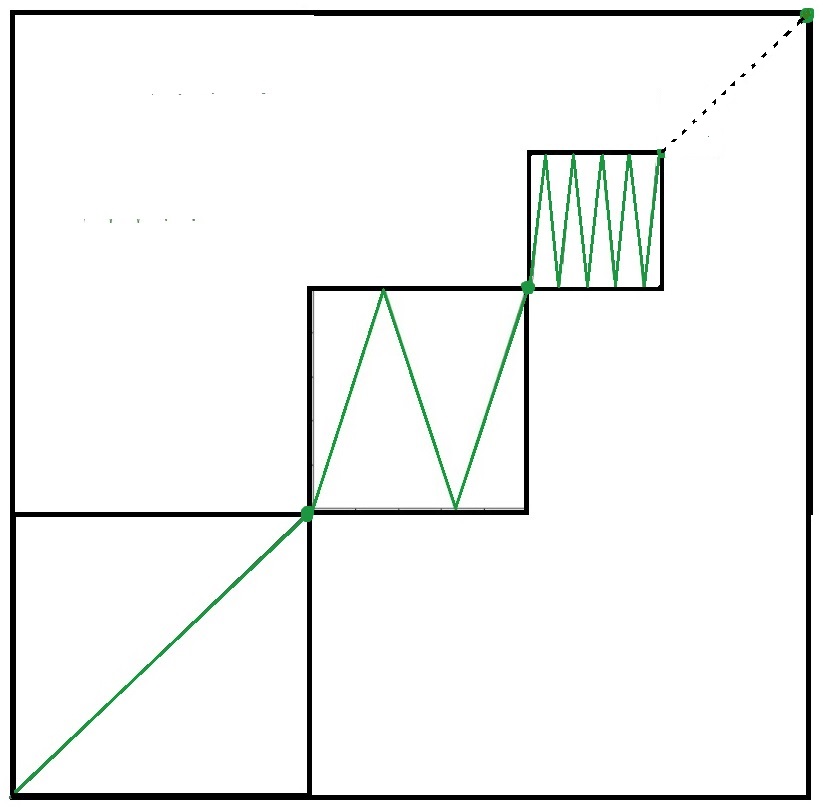}
\caption{Graph of the map $T$.}
\label{caixinhas}
\end{center}
\end{figure}

\end{example}

\begin{example}
    Let $\si_\NN\colon [0,1]^\NN\to[0,1]^\NN$ be the shift map. It was shown in \cite{CPV} that $\cD\equiv1$, so $D_p(X,d,T)\,=\,D_p^{full}(X,d,T)\,=\,[0,1]^\NN.$ In particular, every invariant probability measure maximizes \eqref{eqlocalmdim}.
\end{example}

\subsection{Discontinuous maps}
Another consequence of Theorem~\ref{corodef} is that, since it expresses the metric mean dimension of a map in terms of its induced subshift, we may regard one of the equalities in \eqref{invlimit} as 
a definition of the metric mean dimension of an arbitrary (not necessarily continuous) map. The main difference from the continuous setting is that, for a discontinuous map $T\colon X\to X$, the transition set $\Ga=\graph(T)$ is no longer closed and therefore the sequence space induced by it is not a subshift. However, since the metric mean dimension does not distinguish a set from its closure (see Remark~\ref{mdimclosure}), it is natural to consider the subshift of compact type $\overline{\Ga_\NN}$, which always satisfies $$\overline{\mathrm{mdim}}_M(\overline{\Ga_\NN},d^\NN,\si_\NN)\,=\,\overline{\mathrm{mdim}}_M(\Ga_\NN,d^\NN,\si_\NN)
\quad\text{and}\quad \underline{\mathrm{mdim}}_M(\overline{\Ga_\NN},d^\NN,\si_\NN)\,=\,\underline{\mathrm{mdim}}_M(\Ga_\NN,d^\NN,\si_\NN).$$

\begin{definition}\label{discontdef}
   \emph{ Let $(X,d)$ be a compact metric space, $T\colon X\to X$ be a (not necessarily continuous) map and $\Ga\,=\,\graph(T)$. The upper/lower metric mean dimension of $(X,d,T)$ are given by
   \begin{equation*}
\begin{split}
\overline{\mathrm{mdim}}_M(X,d,T)&
\,=\,\overline{\mathrm{mdim}}_M(\overline{\Ga_\NN},d^\NN,\si_\NN)
\\ \underline{\mathrm{mdim}}_M(X,d,T)&
\,=\,\underline{\mathrm{mdim}}_M(\overline{\Ga_\NN},d^\NN,\si_\NN)
\end{split}
\end{equation*} }
\end{definition}

Let us see an interesting use of this definition. It is known (cf. \cite[Proposition~15.2.13]{KH}) that if $T\colon[0,1]\to[0,1]$ is a piecewise monotone map with finitely many, say $c$, full branches, then \begin{equation}\label{entropy}
    h_{top}(T)\,=\,\log c.
\end{equation} Our next result establishes an analogue formula for the metric mean dimension of piecewise monotone maps with infinitely many full branches. \color{black} In this context, the role of counting branches (which represents the number of critical points of the map) is played by the computation of the box dimension of the (now infinite) set of critical points.

Let $I\subset\RR$ be a compact interval and $A\subset I$ be a compact set defined by  $A=I\setminus \cup_{k\geq1} J_k$, where $(J_k)_k$ are pairwise disjoint open intervals ordered non-increasingly in length such that $|I|=\sum_{k\geq1}|J_k|$. We note that the latter equality implies that $A$ has zero Lebesgue measure; and every compact subset of $\RR$ with zero Lebesgue measure has the previous structure. The intervals $(J_k)_k$ are referred to as the \emph{cut-out sets of $A$} and are deeply related to the box dimension of $A$ (cf. \cite[Propositions 3.6 and 3.7]{Fa2}), which can be any value in $[0,1]$.

Consider a map $T_A\colon I\to I$ satisfying: 
\begin{itemize}
    \item[(C1)] $T_A|_{J_k}$ is strictly monotone, for every $k\geq1$.
    \item[(C2)]  $T_A(J_k)\,=\,\overset{\,\circ}{I}$, for every $k\geq1$.
    \medskip
    \item[(C3)] $T_A|_{J_k}$ is differentiable and $|T'_A|_{J_k}|\,\geq\,\eta\,>\,0$, for some $\eta>0$ and every $k\geq1.$
\end{itemize}
The next result relates the metric mean dimension of $T_A$ with the box dimension of $A$, both with respect to the Euclidean distance $d$. Regarding condition (C3), we refer the reader to Subsection~\ref{C3section}.

\begin{maintheorem}\label{gausslikethm}
    Let $A\subset I$ be a compact subset with zero Lebesgue measure and $T_A\colon I\to I$ be any map satisfying conditions \emph{(C1)-(C3)}. Then \begin{eqnarray*}    \overline{\mathrm{mdim}}_M(I,d,T_A)&=& \overline{\dim}_B(A,d)\\ \underline{\mathrm{mdim}}_M(I,d,T_A)&\leq& \underline{\dim}_B(A,d).
    \end{eqnarray*}
     Moreover, if $\dim_B(A,d)$ exists then \begin{equation}\label{eqgausslike}
     \mathrm{mdim}_M(I,d,T_A)\,=\,\dim_B(A,d).
     \end{equation}
\end{maintheorem}

\smallskip

\begin{example}
    Consider the \emph{Gauss map} 
    \begin{equation*}
    \begin{array}{rccc}
G \colon & [0,1] & \rightarrow & [0,1] \\
& 0  & \mapsto & 0\\
& x\ne0 & \mapsto & \frac{1}{x}\,\text{(mod 1)}
\end{array}
\end{equation*} to which we apply Theorem~\ref{gausslikethm} using $A=\{1/n\,:\,n\in\NN\}\,\cup\,\{0\}$ and $\eta=1.$ Thus, $$\mathrm{mdim}_M([0,1],d,G)\,=\, \dim_B(A,d) \,=\,1/2.$$
\end{example}

\medskip

\begin{example}
    We may also apply Theorem~\ref{gausslikethm} to Young induced transformations. For instance, let $f_\alpha\colon[0,1]\to[0,1]$ be the  Manneville-Pomeau map with parameter $\alpha>0$, defined as  
$$
f_\alpha(x) \,= \left\{ \begin{array}{cc}
   x\,+\,2^\alpha\,x^{\alpha+1}  & \text{if $\,\,0\leq x\leq \frac{1}{2}$} \\
     2x\,-\,1 & \text{if $\,\,\frac{1}{2}< x\leq 1$}
    \end{array} 
   \right..
$$ 
In this case, the Young tower induces a transformation $T_\alpha\colon[\frac{1}{2},1]\to[\frac{1}{2},1]$ (cf. \cite[\S 3.5]{Alves}) which satisfies conditions (C1)-(C3) with $\eta=1$ and $A=\{z_n\,:\,n\in\NN\}\cup\{0\}$, where $$ z_n\,\in\,\Big[ \frac{1}{(n+n_0+1)^{1/\alpha}}\,,\,\frac{1}{(n+n_0)^{1/\alpha}} \Big]$$ for some $n_0\in\NN$ and every $n\in\NN.$ Therefore, $$ 
\mathrm{mdim}_M\big(\big[1/2\,,\,1\big],d,T_\alpha\big)\,=\,\dim_B(A,d) \,=\, \frac{\alpha}{1+\alpha}. $$
\end{example}

Regarding the two previous examples, it was proved in \cite{PW} that they belong in the class of EMR maps (cf \cite[Definition 2.3]{IV}). It is also known that the pressure function $P$ associated to the parameterized geometric potential possesses a unique transition point $$s_\infty\,=\,\sup\{t\geq0\,:\,P(t)=+\infty\}\,=\,\inf\{t\geq0\,:\,P(t)<+\infty\}.$$ This value was computed explicitly in \cite[Theorem 2.11]{IV}, where it was shown to be equal to the box dimension of the corresponding set of discontinuities. By Theorem~\ref{gausslikethm} we further conclude that, for these maps, $s_\infty$ is precisely the metric mean dimension with respect to the Euclidean distance. More precisely:

\smallskip

\begin{corollary}\label{criticalpoint}
    Assume that $\dim_B(A,d)$ exists and that $T_A$ satisfies the additional condition \begin{itemize}
         \item[\emph{(C4)}] $\sup_{k}\,\sup_{x,y\in J_k} \Big| \frac{T'(x)}{T'(y)}\Big|\,<\,+\infty.$
    \end{itemize} Then $$ \mathrm{mdim}_M(I,d,T_A)\,=\,\dim_B(A,d)\,=\, s_\infty.$$
\end{corollary}

Since the transition point $s_\infty$ does not change if we replace $T$ by one of its iterates (cf. \cite[Remark 2.13]{IV}), under the assumptions of Corollary~\ref{criticalpoint} we deduce $$\mathrm{mdim}_M(I,d,T_A)
\,=\,\mathrm{mdim}_M(I,d,T_A^k),\qquad\forall\, k\geq1.$$

\smallskip

\section{Preliminaries}\label{prelim}
Let $(X,d)$ be a compact metric space and $T\colon X \to X$ be a continuous map.

\subsection{Metric mean dimension}

 For each $n \in \mathbb{N}$, define the Bowen metric
$$d_n(x,y) \,=\, \max_{0\,\leq\, j\,\leq\, n-1}\,d(T^j(x),\,T^j(y)) \qquad \forall\, x, y \in X$$
which is equivalent to $d$. We sometimes refer to balls with respect to this metric as dynamical balls.

\smallskip

Given a subset $K$ of $X$ and $\ep > 0$, consider the following minimum
\begin{equation}\label{def:S}
S(K,d,\ep) \,=\, \min \left\{ \ell\colon \,\{U_i\}_{1\,\leq \,i\,\leq \,\ell} \text{ finite open cover of $K$, $\mathrm{diam}(U_i, d) \leq \ep$}\right\}
\end{equation}
and the limit
$$h_\ep(K,d,T) \, = \, \lim_{n\,\to\,+\infty} \frac{1}{n}\log S(K,d_n,\ep),$$
which exists since the sequence $\big(\log S(K,d_n,\ep)\big)_n$ is sub-additive in the variable $n$. Recall that the \emph{topological entropy} is given by $h_{top}(K,T)\,=\,\lim_{\ep\,\to\,0^+}h_\ep(K,d,T).$ 

\begin{definition}\label{mdimT}
\emph{The upper/lower metric mean dimension of $(K,d,T)$ are given, respectively, by
$$\overline{\mathrm{mdim}}_M(K,d,T) \,=\, \limsup_{\ep\,\to\, 0^+}\, \frac{h_\ep(K,d,T)}{\log \,(1/\ep)}$$ and $$\underline{\mathrm{mdim}}_M(K,d,T) \,=\, \liminf_{\ep\,\to\, 0^+}\, \frac{h_\ep(K,d,T)}{\log \,(1/\ep)}.$$}
\end{definition}

\smallskip

A subset $E\subset K$ is said to be $\ep-$separated with respect to the metric $d$ if $d(x,y) \geq \ep$ for every $x, y \in E$; it is $\ep-$spanning with respect to the metric $d$ if for every $x\in K$ there exists some $y\in E$ such that $d(x,y)\leq\ep.$ The notion of upper/lower metric mean dimension can be equivalently defined if one replaces $S(K,d,\ep)$ by either
$$S_1(K,d,\ep)\,=\,\sup\Big\{\#E\colon \,E\subset K \text{ is $\ep-$separated}\Big\}$$
or
$$S_2(K,d,\ep) \,=\,\inf\Big\{\#E\colon \,E\subset K \textrm{ is $\ep-$spanning}\Big\}$$ and replaces their limits in $n$ by either $\limsup_n$ or $\liminf_n.$

\begin{remark}\label{mdimclosure}
   \emph{Since $S_1(K,d,\ep)\,\leq\, S_1(\overline{K},d,\ep)\,\leq\,S_1(K,d,\ep/6)$ for every $K\subset X$ and $\ep>0$, the metric mean dimension does not distinguish a set from its closure. That is, for every $K\subset X$,\begin{eqnarray*}
        \overline{\mathrm{mdim}}_M(\overline{K},d,T) &=&\overline{\mathrm{mdim}}_M(K,d,T) \\ \underline{\mathrm{mdim}}_M(\overline{K},d,T) &=&\underline{\mathrm{mdim}}_M(K,d,T),
    \end{eqnarray*} 
    }
\end{remark} 

\medskip

It is known that the metric mean dimension satisfies the following properties.

\begin{proposition}\label{conj}
     Let $(X_1,d_1)$ and $(X_2,d_2)$ be compact metric spaces and $T_1\colon X_1\to X_1$ and $T_2\colon X_2\to X_2$ be continuous maps.
     
     \smallskip 
     
     \noindent $(a)$ If there exists a surjective Lipschitz map $\pi\colon X_2\to X_1$ such that $T_1 \circ \pi\,=\, \pi\circ T_2$, then \begin{align*}
\overline{\mathrm{mdim}}_M(X_1,d_1,T_1)\,&\leq\,\overline{\mathrm{mdim}}_M(X_2,d_2,T_2)\\ \underline{\mathrm{mdim}}_M(X_1,d_1,T_1)\,&\leq\,\underline{\mathrm{mdim}}_M(X_2,d_2,T_2).
     \end{align*}

     \noindent $(b)$ Given a positive integer $k$, \begin{align*}     \overline{\mathrm{mdim}}_M(X_1,d_1,T^k_1)\,&\leq\,k\,\overline{\mathrm{mdim}}_M(X_1,d_1,T_1)\\ \underline{\mathrm{mdim}}_M(X_1,d_1,T^k_1)\,&\leq\,k\,\underline{\mathrm{mdim}}_M(X_1,d_1,T_1).
     \end{align*}
\end{proposition}

\subsection{Box dimension} Given $K\subset X$, the \emph{upper and lower box dimension of $K$} is given by \begin{eqnarray*}
    \overline{\mathrm{dim}}_B(K,d) &=& \limsup_{\ep\,\to\, 0^+}\, \frac{\log S(K,d,\ep)}{\log \,(1/\ep)}\\ \underline{\mathrm{dim}}_B(K,d) &=& \liminf_{\ep\,\to\, 0^+}\, \frac{\log S(K,d,\ep)}{\log \,(1/\ep)}.
\end{eqnarray*} As happens with the metric mean dimension, this notion may be equivalently defined using spanning/separated subsets of $K$.

\subsection{Katok entropy}
Given an ergodic probability measure $\mu\in\cE_T(X)$, $n\geq1$, $\ep>0$ and $\delta\in(0,1)$, denote $$N_\mu(X,d,T,n,\ep,\delta)\,=\,\inf_B \,\big\{S(B,d_n,\ep)\colon \mu(B)>1-\delta\big\},$$ where the infimum is taken over all measurable subsets $B\subset X$ with measure $\mu$ bigger than $1-\delta$. In fact, the same value is attained if, instead of being just measurable, we assume that $B$ ranges over finite unions of $(n,\ep)-$dynamical balls whose union has measure bigger than $1-\delta.$ We will always consider such $B'$s. The \emph{$\delta-$Katok entropy of $\mu$ at scale $\ep$} is given by $$h_\mu^K(X,d,T,\ep,\delta)\,=\,\limsup_{n\,\to\,+\infty}\frac{1}{n}\log N_\mu(X,d,T,n,\ep,\delta).$$

The following variational principle links Katok entropy and metric mean dimension (see \cite{Shi} and \cite{CL}).

\begin{theorem}\label{pvkatok}
    Let $(X,d)$ be a compact metric space and $T\colon X\to X$ be a continuous map. Then, for every $\delta\in(0,1),$ \begin{eqnarray*}
        \overline{\mathrm{mdim}}_M(X,d,T)\,=\,\limsup_{\ep\,\to\,0^+}\frac{\sup_{\mu\,\in\,\cE_T(X)}h_\mu^K(X,d,T,\ep,\delta)}{\log(1/\ep)} \\ \underline{\mathrm{mdim}}_M(X,d,T)\,=\,\liminf_{\ep\,\to\,0^+}\frac{\sup_{\mu\,\in\,\cE_T(X)}h_\mu^K(X,d,T,\ep,\delta)}{\log(1/\ep)}.
    \end{eqnarray*}
\end{theorem}  

For further use, we also include the following easy consequence of the previous variational principle.

\begin{corollary}\label{aaa}
     Let $(X,d)$ be a compact metric space and $T\colon X\to X$ be a continuous map. Then, \begin{eqnarray*}
         \overline{\mathrm{mdim}}_M(X,d,T)\,=\,\overline{\mathrm{mdim}}_M\big(\bigcap_{k=0}^\infty T^{k}(X),d,T\big) \\ \underline{\mathrm{mdim}}_M(X,d,T)\,=\,\underline{\mathrm{mdim}}_M\big(\bigcap_{k=0}^\infty T^{k}(X),d,T\big).
     \end{eqnarray*}
\end{corollary}

\medskip

\section{Proof of Theorem~\ref{n=z}}\label{proofA}

In this section we show that the unilateral and bilateral subshifts of compact type generated by a closed transition set have the same metric mean dimension.  The next arguments hold for both upper and lower metric mean dimension, so to simplify the notation we will not distinguish them.

In what follows, $\KK=\NN$ or $\ZZ$ and, given a closed set $\Ga\subset X\times X$, we consider the subshift of compact type $$\Ga_\KK\,=\,\{ (x_i)_{i\in \KK}\,\colon (x_i,x_{i+1})\in\Ga,\,\,\forall\, i\in\KK \}$$ endowed with the distance $$d^\KK(x,y)\,=\,\sup_{i\,\in\,\KK}\frac{d(x_i,y_i)}{2^{|i|-1}}.$$

    Denote $Z\,=\,\bigcap_{k=0}^\infty \si_\NN^{k}(\Ga_\NN)$ and keep the notation $d^\NN$ and $\si_\NN$ for the distance and dynamics restricted to $Z$, respectively.  By Corollary~\ref{aaa}, to prove Theorem~\ref{n=z} it is enough to show that
    $$\mathrm{mdim}_M(Z,d^\NN,\si_\NN)\,=\,\mathrm{mdim}_M(\Ga_\ZZ,d^\ZZ,\si_\ZZ).$$
    
    Firstly, note that the projection $\pi\big((x_i)_{i\in\ZZ}\big)\,=\,(x_i)_{i\in\NN}$ is Lipschitz, $\pi\,\circ\,\si_\ZZ\,=\,\si_\NN\,\circ\,\pi$ and $\pi( \Ga_\ZZ)\,=\,Z$. Hence, the restriction $$\pi\colon \Ga_\ZZ\,\longrightarrow\, Z$$ satisfies the conditions in Proposition~\ref{conj}, and therefore $$\mathrm{mdim}_M(Z,d^\NN,\si_\NN)\,\leq\,\mathrm{mdim}_M(\Ga_\ZZ,d^\ZZ,\si_\ZZ).$$

    To show the reverse inequality we will make use of the variational principle provided by Theorem~\ref{pvkatok}. Given a subshift  $Y\subset X^\KK$, a finite set of coordinates $i_1,\cdots,i_k\in\KK$ and a collection $U_1,\cdots,U_k\subset X$ of open subsets, the associated cylinder is defined by
    \begin{equation}\label{cylinderdef}
    C(Y\,;\,i_1,\cdots,i_k\,;\,U_1,\cdots, U_k)\,=\,\big\{(x_i)_{i\in\KK}\in Y\colon x_{i_j}\in U_j,\,\forall j=1,\cdots ,k\big\}.
    \end{equation} These sets are generators of the topology of $Y$ and every dynamical ball is a cylinder. Thus, for any ergodic measure $\nu\in\cE_{\si_\KK}(Y)$, we can compute $$N_\nu(Y,d^\KK,\si_\KK,n,\ep,\delta)\,=\,\inf_B \,\big\{S(B,d_n^\KK,\ep)\colon \nu(B)>1-\delta\big\}$$ by taking the infimum over finite unions of cylinders in $Y$ whose union has measure $\nu$ bigger than $1-\delta$. Note that, for every ergodic measure $\mu\in\cE_{\si_\ZZ}(\Ga_\ZZ)$, its push-forward  $\pi_*\mu\in\cE_{\si_\NN}(Z)$ is ergodic. Moreover, every finite union of cylinders in $Z$ with measure $\pi_*\mu$ bigger than $1-\delta$ has as a pre-image by $\pi$ a finite union of cylinders with measure $\mu$ bigger than $1-\delta$. Consequently, 
    \begin{equation}\label{bbb}
        N_\mu (\Ga_\ZZ,d_\ZZ,\si_\ZZ,n,\ep,\delta) \,=\, \inf_{\substack{ B\, \subset\,\Ga_\ZZ \\\mu(B)>1-\delta}  } S(B,d^\ZZ_n,\ep)\,\leq \,\inf_{\substack{ A\, \subset\, Z\\ \pi_*\mu(A)>1-\delta} } S(\pi^{-1}(A),d_n^\ZZ,\ep).
    \end{equation}

    We proceed by estimating the right hand side in \eqref{bbb}. For each $\ep>0,$ let $\ell=\ell(\ep)$ be a positive integer such that $2^{-\ell}\diam X<\ep$ and $\cU=\{ U_1,...,U_{M(\ep)} \}$ be any open cover of $X$ satisfying $\diam(\cU,d)<\ep.$ Given an open cover $\alpha$ of a set $A\subset Z$ such that $\diam(\alpha,d_n^\NN)\leq\ep$, then $$  \beta = \{ ...\times X \times X \times U_{j_{-\ell}} \times ... \times U_{j_{0}} \times V : V\in\alpha,\,\forall j_i =1,...,M(\epsilon)  \}. $$ is an open cover of $\pi^{-1}(A)\subset \Ga_\ZZ$ satisfying $\diam(\beta,d_n^\ZZ)\leq\ep$ and $\#\beta=\big(M(\ep)\big)^{\ell+1} \#\alpha$. Thus, for every $A\subset Z$ we have $$ S(\pi^{-1}(A),d_n^\ZZ,\ep)\,\leq\, \big(M(\epsilon)\big)^{\ell+1} \,S(A, d_n^\NN,\ep).$$ Combining this information with \eqref{bbb} we get 
\begin{align*}\label{ccc}
        N_\mu^\ZZ (\Ga_\ZZ,d_\ZZ,\si_\ZZ,n,\epsilon,\delta) \,&\leq \,\inf_{\substack{ A\, \subset\, Z\\ \pi_*\mu(A)>1-\delta} }\,S(\pi^{-1}(A),d_n^\ZZ,\ep) \\ & \leq\,  \big(M(\epsilon)\big)^{\ell+1}\, \inf_{\substack{ A\, \subset\, Z\\ \pi_*\mu(A)>1-\delta} }\, S(A,d_n^\NN,\ep)\\ & =\, \big(M(\epsilon)\big)^{\ell+1}\,\, N_{\pi_*\mu}^\NN (Z,d_\NN,\si_\NN,n,\epsilon,\delta).
\end{align*} Hence, for every ergodic probability measure $\mu\in\cE_{\si_\ZZ}(\Ga_\ZZ)$ we have $$h^K_\mu(\Ga_\ZZ,d_\ZZ,\si_\ZZ,\ep,\delta)\,\leq\,h^K_{\pi_*\mu}(Z,d_\NN,\si_\NN,\ep,\delta)$$ and therefore $$\sup_{\mu\,\in\,\cE_{\si_\ZZ}(\Ga_\ZZ)}h^K_\mu(\Ga_\ZZ,d_\ZZ,\si_\ZZ,\ep,\delta)\,\leq\,\sup_{\nu\,\in\,\cE_{\si_\NN}(Z)}h^K_\nu(Z,d_\NN,\si_\NN,\ep,\delta).$$
Applying now Theorem~\ref{pvkatok}, we obtain $$\mathrm{mdim}_M(\Ga_\ZZ,d^\ZZ,\si_\ZZ)\,\leq\,\mathrm{mdim}_M(Z,d^\NN,\si_\NN)$$ and the proof of the proposition is complete. \qed

\medskip

\section{Proof of Theorem~\ref{corodef}}\label{proofB}

In this section we recall Friedland's topologial entropy of a free semigroup action of continuous maps on a compact metric space X, which inspires a corresponding notion of metric mean dimension. Afterwards, we present the proposal in \cite{CRV} for the metric mean dimension of a free semigroup action with respect to a fixed random walk. We proceed by establishing a variational principle connecting these concepts, from which Theorem~\ref{corodef} is a direct consequence.

\subsection{Free semigroup actions}

Let $(X,d)$ be a compact metric space and $G_1\,=\,\{g_1,\cdots,g_p\}$ be a family of continuous maps. Denote by $G$ the free semigroup having $G_1$ as a generator, where the semigroup operation $\circ$ is the composition of maps. Let $\SSS$ be the induced free semigroup action

$$
\begin{aligned}
 \SSS:G\times X \to X\\
(g,x) \mapsto g(x).
\end{aligned}
$$ 

\smallskip

Denoting the index set of $G_1$ by $Y\,=\,\{1,\cdots,p\}$, we endow the product space $Y^\NN$ with the metric $$d'(w,\om)=2^{- \min\{j\,\geq\,1\,:\, w_j\,\ne\, \om_j\}}$$ and consider the skew-product $T_G$ associated to the action $\SSS$: $$
\begin{aligned}
 T_G\colon\,Y^\NN\times X\quad &\to\quad Y^\NN\times X\\
(w,x) \quad&\mapsto\quad (\si(w),g_{w_1}(x)),
\end{aligned} $$ where $w\,=\,(w_1,w_2,\cdots).$

\subsubsection{\emph{\textbf{Friedland's approach}}}

For a finite set of continuous maps $G_1\,=\,\{g_1,\cdots,g_p\}$, we consider the transition set $$\Ga(G_1)\,=\,\bigcup_{i=1}^p \graph(g_i)$$ and the associated subshift of compact type $$\Ga(G_1)_\NN \,=\, \big\{ (x,\,g_{w_1}(x),\cdots,\,g_{w_k}\circ...\circ g_{w_1}(x),\cdots)\,:\,w\,\in\,Y^\NN \big\}. $$
In what follows, the composition operation between two maps will be denoted by their concatenation.


\medskip

\begin{definition}
    \emph{The Friedland topological entropy of $\SSS$ with respect to $G_1$ is defined as $$h^F_{top}(X,\SSS,G_1)\,=\, h(\Ga(G_1)_\NN,\si).$$ }
\end{definition}

Actually, Friedland defined the entropy of $\SSS$ as the infimum of $h^F_{top}(X,\SSS,G_1)$ over all finite sets of generators $G_1$ of $G$. As both definitions that we consider are expressed in terms of a particular set of generators, we keep this dependence and omit $G_1$ from the notation. We now define the corresponding notion for the metric mean dimension.

\begin{definition}
    \emph{The upper/lower Friedland metric mean dimension of $(X,d,\SSS)$ with respect to the set of generators $G_1$ is given by \begin{equation}\label{defsg}
        \begin{split}
            \overline{\mathrm{mdim}}^F_M(X,d,\SSS)\,&=\, \overline{\mathrm{mdim}}_M(\Ga(G_1)_\NN,d^\NN,\si) \\ \underline{\mathrm{mdim}}^F_M(X,d,\SSS)\,&=\, \underline{\mathrm{mdim}}_M(\Ga(G_1)_\NN,d^\NN,\si).
        \end{split}
    \end{equation}}

\end{definition}

\smallskip

 By Theorem~\ref{n=z}, the above definition is independent of whether $\KK\,=\,\NN$ or $\KK\,=\,\ZZ.$ The aim of Theorem~\ref{corodef} is to ensure that the above definition coincides with Lindenstrauss-Weiss' one in the case when $G_1=\{T\}$.

\subsubsection{\emph{\textbf{Carvalho, Rodrigues and Varandas' approach}}}

Now we present the definition of metric mean dimension introduced in \cite{CRV4}. This perspective is inspired by Bufetov (see \cite{Buf}), where one selects randomly which element of $G_1$ will be used to evolve time on each step.  

Following \cite{CRV4}, we shall code different concatenations of elements in $G_1$ by distinct sequences of symbols. Even though an element of the group may be generated by different combinations of generators, we will simply consider different concatenations instead of the elements in the group they create, since we do not make use of the group structure.

For every $ w\in Y^\NN$ and $n\geq1$, consider the metric in $X$ given by $$ d_{w,n}(x,y)\,= \,\max_{0\leq j\leq n} d(g_{w_j} g_{w_{j-1}}...g_{w_1} (x),g_{w_j} g_{w_{j-1}}...g_{w_1} (y)). $$
Fix a \emph{random walk}, that is, a shift invariant Borel probability measure $\PP\in\cP_\si(Y^\NN)$. The \emph{$\ep-$entropy} of $(X,d,\SSS,\PP)$ is given by $$h_\ep(X,d,\SSS,\PP)\,=\, \limsup_{n\,\to\,+\infty}\frac{1}{n}\log \int_{Y^\NN}S(X,d_{w,n},\ep)d\PP(w).$$ The \emph{topological entropy of $\SSS$ with respect to the set of generators $G_1$} (cf. \cite{CRV4}) is given by $$h_{top}(X,\SSS)\,=\,\sup_{\PP}\,\lim_{\ep\,\to\,0^+}\,h_\ep(X,d,\SSS,\PP).$$

\begin{definition}\label{mdimsg}
\emph{The upper/lower metric mean dimension of $(X,d,\SSS)$ with respect to the set of generators $G_1$ and the random walk $\PP$ are given, respectively, by
\begin{eqnarray*}
    \overline{\mathrm{mdim}}_M(X,d,\SSS,\PP) \,& =&\, \limsup_{\ep\,\to\, 0^+}\, \frac{h_\ep(X,d,\SSS,\PP)}{\log \,(1/\ep)} \\ \underline{\mathrm{mdim}}_M(X,d,\SSS,\PP) \,& = &\, \liminf_{\ep\,\to\, 0^+}\, \frac{h_\ep(X,d,\SSS,\PP)}{\log \,(1/\ep)}.
\end{eqnarray*}}
\end{definition}

Once again, one can replace $S$ in the definition by either $S_1$ or $S_2$ as in Definition~\ref{mdimT}. 

\begin{definition}
   \emph{Given a compact metric space $M$, we say that $\nu\in\cP(M)$ is homogeneous if it is a fully supported doubling measure, namely, $$\exists L>0\colon\qquad \nu(B(x,\ep))\leq L \textrm{ }\nu(B(y,\ep/2)),\qquad\forall x,y\in M,\,\,\forall\ep>0. $$}
\end{definition}

Note that if a random walk $\PP\,\in\,\cP_\si(Y^\NN)$ is homogeneous, then $$\PP(B(w,2^{-k+1})) \,\geq\, (1/L)^k \,\PP(B(w,1/2)) \,=\,(1/L)^k,$$ for every $w\in Y^\NN$ and $k\geq1$, where $$B(w,2^{-k+1})\,=\,\{\om\in Y^\NN\colon \om_i=w_i, i=1,...,k\}.$$ In what follows, we adopt the notation $\overline{w_1...w_k}\,=\,B(w,2^{-k+1}).$

\subsubsection{\emph{\textbf{Variational principle}}}

Our main result in this section states that the two aforementioned notions of metric mean dimension of finitely generated free semigroup actions are equal. Theorem~\ref{corodef} corresponds to the particular case of $G_1=\{T\}$.

\begin{theorem}\label{sg} 
Let $(X,d)$ be a compact metric space and $G_1\,=\,\{g_1,\cdots,g_p\}$ be a collection of continuous maps in $X$. Then,
  \begin{eqnarray}\label{eq1}\label{eq1}
      \overline{\mathrm{mdim}}^F_M(X,d,\SSS) \,&=& \,\max_\PP\, \overline{\mathrm{mdim}}_M(X,d,\SSS,\PP) \\ \nonumber\underline{\mathrm{mdim}}^F_M(X,d,\SSS) \,&=& \,\max_\PP \,\underline{\mathrm{mdim}}_M(X,d,\SSS,\PP).
  \end{eqnarray}
       Moreover, the previous maxima are attained at \emph{every} homogeneous random walk $\PP$ and we have \begin{equation}\label{eqskew}
      \overline{\mathrm{mdim}}^F_M(X,d,\SSS) \,= \, \overline{\mathrm{mdim}}_M(X\times I^{\NN},\,d\times d',\,T_G).
  \end{equation}
\end{theorem}

\medskip

The equality \eqref{eqskew} is a direct consequence of the fact that the maximum in \eqref{eq1} is attained by any fully supported Bernoulli measure $\PP$ - which is homogeneous - combined with \cite[Corollary~III]{CRV4}. In fact, the choice of metric in $Y^\NN$ considered in \cite{CRV4} is different from our choice $d'$,  but since both metrics are uniformly equivalent so are their products with $d$ and therefore their corresponding metric mean dimensions of the skew-product coincide.

\begin{remark}
    \emph{We observe that a statement analogous to Theorem~\ref{sg} for the topological entropy of a semigroup action does \emph{not} hold. Friedland showed \cite[Lemma 3.6]{Fried} that if $T$ is a non-involutive homeomorphism ($T^2\ne Id$) of a compact metric space, then $h_{top}^F(\SSS)\geq\log2.$ If we consider the phase space to be the unit circle, both $T$ and $T^{-1}$ have zero topological entropy, and since compositions of them only represent delay in time, we have $h_{top}(\SSS,\PP)\,=\,0$ for every $\PP.$ So $$h_{top}^F(\SSS) \,>\, \max_{\PP} \,h_{top}(\SSS,\PP).$$
    Actually, the Friedland subshift is a Lipschitz factor of the skew-product map - in particular, $h_{\ep}(\Ga_\NN,d^\NN,\si)\leq h_{\ep}(T_G)$ - and the entropy of the action relates to the one of the skew product by Bufetov's formula (cf. \cite{Buf}): if $\PP_p$ is the symmetric Bernoulli measure in the $p$ symbols of $Y$, then $$h_{top}(\SSS,\PP_p)\,=\, h_{top}(T_G)\,-\,\log p.$$ Whereas, in the case of metric mean dimension, the normalization $\log(1/\ep)$ in its definition makes the second term in the right hand side  vanish and we obtain \eqref{eq1} and \eqref{eqskew}. 
    }
\end{remark}

\begin{proof}[Proof of Theorem~\ref{sg}]
    
 We start by proving that the Friedland metric mean dimension is an upper bound for the one induced by any random walk. Let $w\in Y^\NN$. For each $n\geq1,$ consider $E\subset X$ a maximal $\ep-$separated subset with respect to the metric $d_{w,n}.$ That is, for every $x,y\in E$ we have
$$\ep<\max\big\{ d(x,y), d(g_{w_1}(x),g_{w_1}(y)),\cdots,d(g_{w_n}... g_{w_1}(x),g_{w_n}...g_{w_1}(y)) \big\}.$$Let us denote simply $\Ga\,=\,\Ga(G_1)$. Then, the elements of $\Ga_\NN$ \begin{align*}
      \big(&x,g_{w_1}(x),\cdots,g_{w_n}... g_{w_1}(x),\footnotesize{\text{any admissible tail}}\big) \\\big(&y,g_{w_1}(y),\cdots ,g_{w_n}...g_{w_1}(y),\footnotesize{\text{any admissible tail}}\big)
\end{align*}

\smallskip 

\noindent are $(\si,n,\ep)$ separated. Hence, for every $\PP\in\cP_{\si}(Y^\NN)$ we have 
\begin{align*}
    h_\ep(X,d,\SSS,\PP) \,&=\,\limsup_{n\,\to\,+\infty}\,\frac{1}{n}\,\log \int_{Y^\NN}S_1(X,d_{w,n},\ep)\,d\PP(w) \\&\leq\, \limsup_{n\,\to\,+\infty}\,\frac{1}{n}\,\log S_1(\Ga_\NN,d^\NN_n,\ep)\\&=\, h_\ep(\Ga_\NN,d^\NN,\si)
\end{align*}

Now let us show that the maximum is attained by any homogeneous $\PP\in\cP_{\si}(Y^\NN)$. Take one such $\PP$. Then, for every $w\in Y^\NN$ and $k\geq1$, we have $$\PP(\overline{w_1...w_k}) \,\geq\,(1/L)^k.$$ Given $n\geq1$ and $\ep>0$, let $E=\{x_1,x_2,...,x_{N}\}$ be a maximal $(\si,n,\ep)-$separated subset in $\Ga_\NN$, that is, $N=S_1(\Ga_\NN,d_n^\NN,\ep)$. Take $\ell=\ell(\ep)$ such that $2^{-\ell}\diam(X) \leq \ep.$ By the Pigeonhole Principle, there exists a subset $A\subset E$ such that 
\begin{itemize}
    \item all the transitions until the coordinate $n+\ell$ are given by the same $(n+\ell)-$tuple, which we denote by $(g_{w_1},g_{w_2},...,g_{{w_{n+\ell}}})$;
    \item $a:=|A|\geq \frac{N}{p^{n+\ell}}.$ 
\end{itemize}
 Note that, by definition of $\ell$, the property of being separated only depends on the first $n+\ell$ coordinates of the elements of $E$. In particular, the first coordinate of all the elements in $A$ are necessarily different from each other. Reordering if necessary, we denote the elements of $A$ by $$ x_1 = \big(x^1,\,g_{w_1}(x^1),\,g_{w_2} g_{w_1}(x^1),\,\cdots,g_{w_{n+\ell}}...g_{w_1}(x^1),\footnotesize{\text{any admissible tail}}\big) $$ 
 $$ x_2 = \big(x^2,\,g_{w_1}(x^2),\,g_{w_2} g_{w_1}(x^2),\,\cdots,g_{w_{n+\ell}}...g_{w_1}(x^2),\footnotesize{\text{any admissible tail}}\big) $$ 
 $$\vdots$$ $$ x_a = \big(x^a,\,g_{w_1}(x^a),\,g_{w_2} g_{w_1}(x^a),\,\cdots,g_{w_{n+\ell}}...g_{w_1}(x^a),\footnotesize{\text{any admissible tail}}\big) .$$ Then, the set $\{x^1,...,x^a\}\subset X$ is $\ep-$separated with respect to the metric $d_{\om,n+\ell}$ for every $\om\in \overline{w_1...w_{n+\ell}}\subset Y^\NN.$ In particular, the balls of radius $\ep/3$ with respect to $d_{\om,n+\ell}$ centered at the elements of $A$ are pairwise disjoint. Hence, $$S_2(X,d_{\om,n},\ep/3)\,\geq\,a\,\geq\,\frac{N}{p^{n+\ell}}$$ for every $\om\in \overline{w_1...w_{n+\ell}}$. So,
\begin{align*}
    \int_{Y^\NN}S_2(X,d_{\om,n},\ep/3)\,d\PP(\om)\, &\geq\, \int_{ \overline{w_1...w_{n+\ell}}}S_2(X,d_{\om,n},\ep/3)\,d\PP(\om) \\  &\geq\,\frac{N}{p^{n+\ell}}\,\,\PP(\overline{w_1...w_{n+\ell}}) \\&\geq\,\frac{N}{p^{n+\ell}}\,(1/L)^{n+\ell}.
\end{align*}
Therefore, \begin{align*}
    h_{\ep/3}(X,d,\SSS,\PP) \,&=\,\limsup_{n\,\to\,+\infty}\,\frac{1}{n}\,\log \int_{Y^\NN}S_2(X,d_{\om,n},\ep/3)\,d\PP(\om) \\&\geq\, h_\ep(\Ga_\NN,d^\NN,\si)\,-\log(p\,L).
\end{align*}
We conclude after dividing by $\log(1/\ep)$ and making $\ep\to0^+$.
\end{proof}

\subsection{Zero complexity maps may generate positive metric mean dimension}

In this subsection we will make use of Theorem~\ref{sg} to show that two maps with zero metric mean dimension may generate a free semigroup action with positive metric mean dimension. Let $X=[0,1]$ endowed with the Euclidean distance $d$.

\begin{proposition}\label{mdim1/2}
    There exist two zero entropy continuous maps $g_1,g_2\colon [0,1]\,\to\,[0,1]$ satisfying $$\mathrm{mdim}^F_M([0,1],d,\SSS)\,=\,\frac{1}{2}.$$ In particular, the topological entropy of the action is infinite.
\end{proposition}

\begin{proof}

Let $T\colon[0,1]\to[0,1]$ be the map defined in Example~\ref{defcaixinhas}. We consider the maps $g_1,g_2\colon [0,1]\,\to\,[0,1]$ $$
g_1(x)\,=\,\left \{ \begin{matrix} 0 & \text{if }\,0\leq x\leq\frac{1}{2} \\ \frac{T(2x-1)}{2} &\text{if }\, \frac{1}{2}\leq x\leq1 \end{matrix} \right. \qquad\quad g_2(x)\,=\,\left \{ \begin{matrix} x+\frac{1}{2} &\text{if }\, 0\leq x\leq\frac{1}{2} \\ 1 &\text{if }\, \frac{1}{2}\leq x\leq1 \end{matrix} \right.  .
$$ 

\medskip

\noindent The transition set $\Ga$ which generates the Friedland subshift is illustrated in Figure~\ref{figure1}.

\begin{figure}[!htb]
\begin{center}
\includegraphics[scale=0.5]{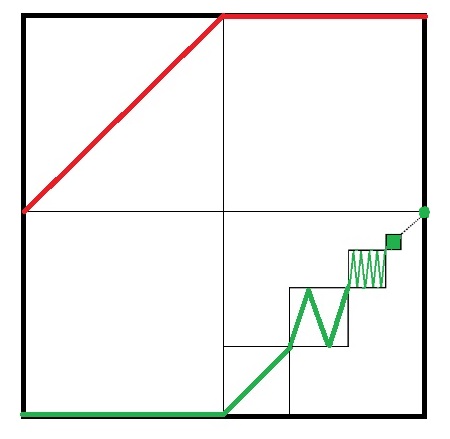}
\caption{$\Ga\,=\, \color{ForestGreen}\graph(g_1)\,\color{black}\cup\,\color{red}\graph(g_2)\color{black}$.}
\label{figure1}
\end{center}
\end{figure}

\noindent Since $g_1^2\equiv 0$ and $g_2^2\equiv1$, it is clear that their topological entropy is zero, and so both maps have zero metric mean dimension. We also observe that - denoting the Bowen metric of time $k$ with respect to a map $\phi$ by $d_k^\phi$ - we have (see Figure~\ref{figure2}) $$S_1([0,1],d^{g_1\circ g_2}_{n},\ep)\,=\,S_1([0,1],d^{g_2\circ g_1}_{n},\ep)\,=\, S_1([0,1],d^{T}_{n},2\ep)\,+\, \Big\lfloor\frac{1}{2\ep}\Big\rfloor.$$

\begin{figure}[!htb]
\begin{center}
\includegraphics[scale=0.3]{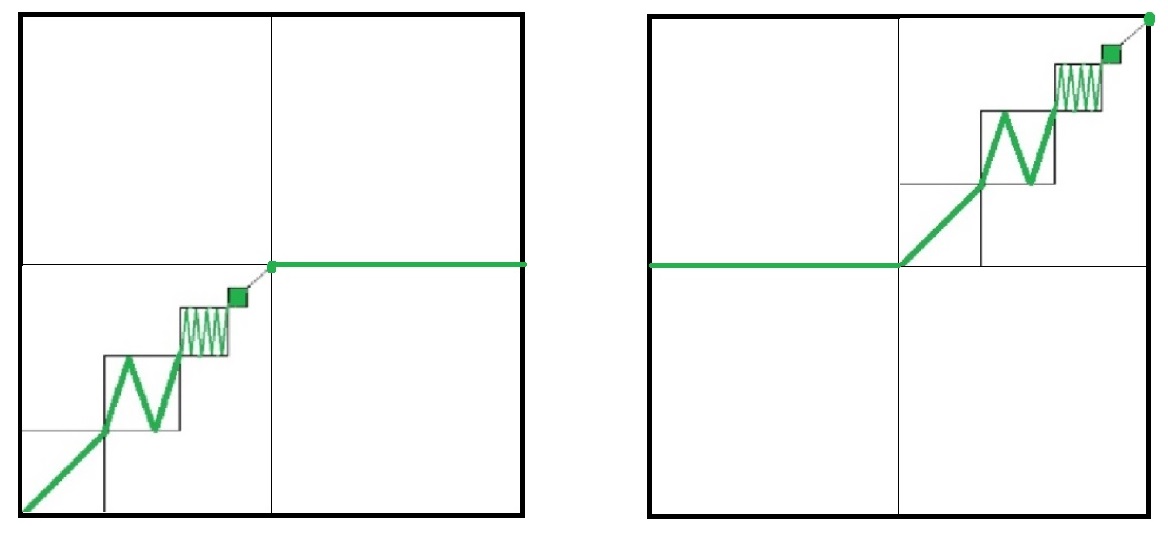}
\caption{Graphs of $g_1\circ g_2$ and $g_2\circ g_1$, respectively.}
\label{figure2}
\end{center}
\end{figure}

 In order to prove the lower bound $$\mathrm{mdim}^F_M([0,1],d,\SSS)\,\geq\,\frac{1}{2},$$ we consider the random walk $\PP\,=\,\frac{1}{2}\big( \delta_{(12)^\NN}\,+\,\delta_{(21)^\NN} \big).$ Fix $w=(12)^\NN$ (the case $w=(21)^\NN$ is analogous). Since $$d_{w,2n}(x_0,y_0)\,\geq\,d_n^{g_2\circ g_1}(x_0,y_0),\qquad\forall x_0,y_0\,\in\,[0,1],\,\,\forall n\geq1,$$ we have $$S_1([0,1],d_{w,2n},\ep)\,\geq\,S_1([0,1],d^{g_2\circ g_1}_{n},\ep)\,=\,S_1([0,1],d^{T}_{n},2\ep)\,+\,\Big\lfloor\frac{1}{2\ep}\Big\rfloor$$ and so $$h_{\ep}([0,1],d,\SSS,\PP) \,=\,\limsup_{n\,\to\,+\infty}\,\frac{1}{2n}\,\log \int_{\{0,1\}^\NN}S_1([0,1],d_{w,2n},\ep)\,d\PP(w)\,\geq \, \frac{h_{2\ep}([0,1],d,T)}{2},$$ which implies that $$\frac{1}{2}\,\leq\,\underline{\mathrm{mdim}}_M([0,1],d,\SSS,\PP)\,\leq\,\underline{\mathrm{mdim}}^F_M([0,1],d,\SSS).$$

\medskip

Now we proceed to prove the converse inequality. By \cite[Theorem 1.1]{Tsu}, for every $\delta>0$ we have 
\begin{equation}\label{Tsuk}    \overline{\mathrm{mdim}}_M(\Ga_\ZZ,d^\ZZ,\si_\ZZ)\,=\,\limsup_{\ep\,\to\,0^+}\frac{1}{\log(1/\ep)}\,\sup_{x\,\in\,\Ga_\ZZ} h_\ep(B_\delta(x,d_\ZZ^\ZZ), d^\ZZ,\si_\ZZ),
\end{equation} where $x=(x_j)_{j\in\ZZ}$ $$B_\delta(x,d_\ZZ^\ZZ)\,=\, \big\{y\,\in\,\Ga_\ZZ\,:\, d^\ZZ(\si^j(x),\si^j(y))\,\leq\,\delta,\,\,\forall j\in\ZZ\big\}\,=\,\big\{y\,\in\,\Ga_\ZZ\,:\, d(x_j,y_j)\,\leq\,\delta,\,\,\forall j\in\ZZ\big\}.$$

\smallskip

\noindent Fix $0\,<\,\delta\,<\, |J_1|/6.$ Since $\delta\, \ll \,(d\times d)(\graph(g_1),\graph(g_2))$, every $y\in B_\delta(x,d_\ZZ^\ZZ)$ must be prescribed by the same transitions as $x$, namely, there exists $w(x)\in\{0,1\}^\ZZ$ such that $y_{k+1}\,=\, g_{w(x)_k}(y_k)$ for every $k\in\ZZ$ and $y\in B_\delta(x,d_\ZZ^\ZZ)$. Let us estimate $ h_\ep(B_\delta(x,d_\ZZ^\ZZ), d^\ZZ,\si_\ZZ)$ separately for each type of $w(x)$.

\medskip 

\noindent\textbf{Case 1:} $(21)^\ZZ\ne w(x)\ne (12)^\ZZ$ 

\smallskip

\noindent Let $k_0\in\ZZ$ be such that $w(x)_{k_0}\,=\,w(x)_{k_0-1} $. Then we have $y_{k}\,=\,x_k$ for every $k\geq k_0$ and $y\in B_\delta(x,d_\ZZ^\ZZ)$. Then, \begin{equation}\label{1}
    h_\ep(B_\delta(x,d_\ZZ^\ZZ), d^\ZZ,\si_\ZZ)\,=\,0.
\end{equation}

\bigskip

\noindent\textbf{Case 2:} $w(x)= (\cdots ,1,2\,;1,2,1,2,\cdots)$

\medskip

\noindent $(a)$ $x_0\in [0,1/2 +2\delta]$. 

In this case, $\pi_0\Big(B_\delta(x,d_\ZZ^\ZZ)\Big)\subset [0,1/2 + 3\delta]$. Given $y\in B_\delta(x,d_\ZZ^\ZZ)$, we have two possibilities \begin{itemize}
    \item If $y_0\in[0,\frac{1}{2}]$, then $y_1=g_1(y_0)=0.$ This implies that $y_{2k-1}=0$ and $y_{2k}=1/2,$ for every $k\geq1.$ Then, $\pi_\NN\Big(\pi_0^{-1}[0,\,1/2]\,\cap\,B_\delta(x,d_\ZZ^\ZZ)\Big)\,=\, (0,1/2,0,1/2,...)$.
    \smallskip
    \item If $y_0\in(1/2,1/2\,+\,3\delta]$, then $y_1=g_1(y_0)=y_0 \,-\,1/2$ and $y_2=g_2(y_1)=y_0$ This implies that $y_{2k-1}=y_0\,-\,1/2$ and $y_{2k}=y_0,$ for every $k\geq1.$ Hence, $$\pi_\NN\Big(\pi_0^{-1} (1/2,1/2\,+\,3\delta]\,\cap\,B_\delta(x,d_\ZZ^\ZZ)\Big)\,\subset\, \{(y_0\,-\,1/2,y_0,y_0-1/2,y_0,...)\,:\,y_0\in (1/2,1/2\,+\,3\delta] \}.$$
\end{itemize}
Thus, $$\pi_\NN\Big(B_\delta(x,d_\ZZ^\ZZ)\Big)\,\subset\, \{(y_0\,-\,1/2,\,y_0,\,y_0\,-\,1/2,\,y_0,...)\,:\,y_0\in [1/2,\,1/2\,+\,2\delta] \},$$ which implies that \begin{equation*}
    h_\ep(B_\delta(x,d_\ZZ^\ZZ), d^\ZZ,\si_\ZZ)\,=\,0.
\end{equation*}

\medskip

\noindent $(b)$ $x_0\in [1/2 +2\delta,1]$. 

\smallskip

\noindent Take $\ell=\ell(\ep)$ such that $2^{-\ell}\leq\ep$, assuming without loss that $\ell$ is even. Given $y,z\in B_\delta(x,d_\ZZ^\ZZ)$, we have $d(y_{2k},z_{2k})\,=\,d(y_{2k-1},z_{2k-1})$ for every $k\in\ZZ$. If $d_{2n}^\ZZ(y,z)>\ep$, we get $d_{n+\ell+1}^{g_2\circ g_1}(y_{-\ell},z_{-\ell})>\ep$, since the separation between $y$ and $z$ must occur between the coordinates $ - \ell$ and $\ell+2n$ and we may restrict to even times. In particular, this implies that $$S_1(B_\delta(x,d_\ZZ^\ZZ),d_{2n}^\ZZ,\ep)\,\leq\, S_1([0,1],d_{n+\ell+1}^{g_2\circ g_1},\ep)\,=\, S_1([0,1],d_{n+\ell+1}^{T},2\ep)\,+\,\Big\lfloor\frac{1}{2\ep}\Big\rfloor,\qquad \forall n\geq1.$$ Hence $$ h_\ep(B_\delta(x,d_\ZZ^\ZZ), d^\ZZ,\si_\ZZ)\,\leq\,\frac{h_{2\ep}([0,1],d,T)}{2}.$$

\smallskip Bringing together sub-cases $(a)$ and $(b)$ we conclude that for every $x\in\Ga_\ZZ$ such that $w(x)\,=\, (\cdots ,1,2\,;1,2,1,2,\cdots)$ one has 
\begin{equation}\label{2}
     h_\ep(B_\delta(x,d_\ZZ^\ZZ), d^\ZZ,\si_\ZZ)\,\leq\,\frac{h_{2\ep}([0,1],d,T)}{2}.
\end{equation}

\medskip

\noindent\textbf{Case 3:} $w(x)= (\cdots ,1,2,1\,;2,1,2,\cdots)$

\noindent Since $\si_\ZZ$ is a homeomorphism, $\si(w(x))\,=\, w(\si_\ZZ(x))$ and $B_\delta(\si_\ZZ(x),d_\ZZ^\ZZ)\,=\,\si_\ZZ\Big(B_\delta(x,d_\ZZ^\ZZ)\Big)$, we have \begin{align}\label{3}
     h_\ep(B_\delta(x,d_\ZZ^\ZZ), d^\ZZ,\si_\ZZ)\,& \nonumber=\,  h_\ep\Big(\si_\ZZ^{-1}(B_\delta(\si_\ZZ(x),d_\ZZ^\ZZ), d^\ZZ,\si_\ZZ\Big)\\& =\,\nonumber h_\ep\Big(B_\delta(\si_\ZZ(x),d_\ZZ^\ZZ), d^\ZZ,\si_\ZZ\Big)\\& \leq\,\frac{h_{2\ep}([0,1],d,T)}{2},
\end{align}
where the inequality follows from Case 2. Combining \eqref{Tsuk}, \eqref{1},\eqref{2} and \eqref{3} we deduce that 
 $$\overline{\mathrm{mdim}}^F_M([0,1],d,\SSS)\,=\,\overline{\mathrm{mdim}}_M(\Ga_\ZZ,d^\ZZ,\si_\ZZ)\,\leq\,\frac{1}{2}.$$

\smallskip 

\noindent This finishes the proof of Proposition~\ref{mdim1/2}.
\end{proof}

    Proposition~\ref{mdim1/2} seems optimal. If two maps have zero complexity (in terms of metric mean dimension) but the associated action does not, the richness of the orbit structure can only be revealed after a composition of the two maps. This indicates that we need twice the time to see this maximal value, which is always bounded by the dimension of the phase space. This motivates the following

\begin{question}
    Let $(X,d)$ be a compact metric space and $g_1,g_2\colon X\,\to\,X$ be continuous maps with zero metric mean dimension. Is it true that $$\overline{\mathrm{mdim}}^F_M(X,d,\SSS)\,\leq\,\frac{\overline{\dim}_B(X,d)}{2}\,\,\,?$$
\end{question}

\section{Proof of Corollary~\ref{localmdim}}\label{proofcor}

Recall that, given $p\in\Ga_\ZZ$ and $q\in X$, \begin{eqnarray*}
    \cD_\Ga(p)&=&\,  \inf \,\Big\{\overline{\mathrm{mdim}}_M(\overline{U},d^\ZZ,\si_\ZZ)\colon \,U\subset\Ga_\ZZ \textrm{ is an open neighbourhood of p} \Big\} \\ \cD(q)&=&\,  \inf \,\Big\{\overline{\mathrm{mdim}}_M(\overline{U},d,T)\colon \,U\subset X \textrm{ is an open neighbourhood of }q \Big\}
\end{eqnarray*} 

Let $\mu\in\cP_{\si_\ZZ}(\Ga_\ZZ)$ be such that every $x\,=\,(x_i)_{i\in\ZZ}\,\in\,\supp(\mu)$ satisfies \begin{equation*}
    \begin{split}
        \cD_\Ga(x)\, &=\, \overline{\mathrm{mdim}}_M(\Ga_\ZZ,d^\ZZ,\si_\ZZ) 
    \end{split}
\end{equation*} which always exists by \cite[Theorem E]{CPV}. Fix one such $x.$ Then, for any $\ga>0$, we have $$\overline{\mathrm{mdim}}_M(\overline{U_\ga},d^\ZZ,\si_\ZZ)\,=\,\overline{\mathrm{mdim}}_M(\Ga_\ZZ,d^\ZZ,\si_\ZZ),$$ where $U_\ga\,=\,\{ y\,\in\,\Ga_\ZZ\, : \,y_1\,\in\,B_\ga(x_1) \}.$  Given $\ep>0$, let $\ell=\ell(\ep)$ be a positive integer satisfying $2^{-\ell}\diam(X,d)\leq\ep.$ For $n\geq1$, consider a maximal $(n,\ep)-$separated subset $E\subset U_\ga.$ Then, there exist $z^1,...,z^a\in E$ such that \begin{itemize}
    \item $d(z^{m_1}_j,z^{m_2}_j)\leq\ep$, for every $j=-\ell,...,0$ and $m_1\ne m_2$
    \item $a\,\geq\, \# E\,/\, S(X,d,\ep)^{\ell+1}.$
\end{itemize} Hence $$S_1(B_\ga(x_1),d_n,\ep)\,\geq\,\# E\,/\, S(X,d,\ep)^{\ell+1}\,=\, S_1(\overline{U_\ga},d^\ZZ_n,\ep)\,/\, S(X,d,\ep)^{\ell+1}, $$ and so  \begin{equation*}
        \cD(x_1)\, \,=\, \lim_{\ga\to0^+}\,\overline{\mathrm{mdim}}_M(\overline{B_\ga(x_1)},d,T)\,\geq\,\overline{\mathrm{mdim}}_M(\Ga_\ZZ,d^\ZZ,\si_\ZZ) \, =\, \overline{\mathrm{mdim}}_M(X,d,T),
\end{equation*}
where the last equality is due to Theorem~\ref{corodef}. On the other hand, by definition, the reverse inequality always holds. Hence, $$\cD(x_1)\,=\,\overline{\mathrm{mdim}}_M(X,d,T).$$ Now, since $\supp(\mu)\,\subset\,\Ga_\ZZ$ is $\si_\ZZ-$invariant, we may carry on the same reasoning with $x_1$ replaced by $x_{1+j}=(\si_j(x))_1$, obtaining \begin{equation}\label{Diterate}
    \cD(T^j(x_1))\,=\,\cD(x_{1+j})\,=\,\overline{\mathrm{mdim}}_M(X,d,T)\qquad \forall j\geq0.
\end{equation}
Let us consider the empirical measures along the $T-$orbit of $x_1$. For each $k\geq1$, denote $\nu_k\,=\,\frac{1}{k}\sum_{j=0}^{k-1}\delta_{T^j(x_1)}$, which by \eqref{Diterate} satisfies $\int\cD\,d\nu_k\,=\,\overline{\mathrm{mdim}}_M(X,d,T)$. Due to the upper-semicontinuity of $\cD$, any weak* limit $\nu$ of $(\nu_k)_k$ is $T-$invariant and satisfies $$\int\cD\,d\nu\,=\,\overline{\mathrm{mdim}}_M(X,d,T). $$ 

In general, the previous maximizing measure is not ergodic. Yet, as $\mu \mapsto \int \mathcal{D} d\mu$ is affine, the maximum in Corollary~\ref{localmdim} is also attained at some ergodic probability measure.

Finally, observe that given any measure $\mu \in \cP_T(X)$ satisfying
$$\overline{\mathrm{mdim}}_M(X,d,T) \,=\, \int \mathcal{D} \,d\mu,$$ since $\mathcal{D} \leq \overline{\mathrm{mdim}}_M(X,d,T)$, there exists a full measure set $Z \subset X$ such that $\mathcal{D}|_Z \equiv  \overline{\mathrm{mdim}}_M(X,d,T).$ Therefore, since $\supp(\mu) \subset \overline{Z}$ and $\mathcal{D}$ is upper semi-continuous, we deduce that $\mathcal{D}|_{\supp(\mu)} \equiv \overline{\mathrm{mdim}}_M(X,d,T).$ \qed

\medskip

\section{Spectral radius bounds for subshifts of compact type}\label{specradiussection}

In what follows, we recall from \cite{Fried} an upper bound for the $\ep-$entropy of subshifts of compact type in terms of the spectral radius of a suitable matrix. Fix a compact metric space $(X,d)$ and a subset $\Ga\,\subset\,X\times X.$

Given $\ep>0$, let $\cU(\ep)=\{ U_1,...,U_{M(\ep)} \}$ be an open $\ep-$cover of $X.$ Consider the $M(\ep)\times M(\ep)$ matrix $\Ga^\ep$ given by
\begin{equation*}\label{gammaepgeneral}
    \Ga_{i,j}^\epsilon = \left \{ \begin{matrix} 1 & \textrm{if } \big( U_i \times U_j \big) \cap \Ga \ne\emptyset \\ 0 & \textrm{otherwise.} \end{matrix} \right. 
\end{equation*} 
By counting arguments similar to the ones used for subshifts of finite type, we know that $||(\Ga^\epsilon)^k||$ many cylinders composed by elements of $\cU(\ep)$ of size $k$ are enough to cover $\Ga_\NN$, where $||\cdot||$ denotes the sum norm in the space of $M(\ep)\times M(\ep)$ matrices. This estimate yields the following equivalent formulation of \cite[Lemma 4.1]{Fried}, where $r(A)$ stands for the spectral radius of the matrix $A$.

\begin{lemma}\label{specradgeneral}
   Let $(X,d)$ be a compact metric space and $\Ga\subset X\times X$. Then $$h_\ep(\Ga_\NN,d^\NN,\sigma_\NN) \leq \log r(\Ga^\epsilon).$$
\end{lemma}

In the case $X\,=\,[0,1]$ (and also in higher dimensions), the matrix $\Ga^\ep$ can be replaced by another one constructed by using intersections with an $\ep-$grid instead of an open cover. More precisely, given $\ep>0$ and $i,j=1,...,\lceil1/\ep \rceil$, we define \begin{equation}\label{gammaep}
    \Ga_{i,j}^\epsilon = \left \{ \begin{matrix} 1 & \textrm{if } \big( [(i-1)\ep,i\ep] \times[(j-1)\ep,j\ep] \big) \cap \Ga \ne\emptyset \\ 0 & \textrm{otherwise.} \end{matrix} \right.
\end{equation}
In the remaining of this paper the notation $\Ga^\ep$ will always refer to \eqref{gammaep}, for which we have the following reformulation of Lemma~\ref{specradgeneral}.

\begin{lemma}\label{specrad}
   Let $d$ be the Euclidean distance in $[0,1]$ and $\Ga\subset [0,1]^2$. Then $$h_\ep(\Ga_\NN,d^\NN,\sigma_\NN) \leq \log r(\Ga^\epsilon).$$
\end{lemma}

The previous lemmas will be very useful tools to estimate the metric mean dimension when combined with the following property of complex matrices.

\begin{theorem}[Gershgorin Circle Theorem, \cite{Gers}]\label{circlethm}
    Let $A\,=\,(a_{ij})_{i,j}$ be a complex $n \times n$ matrix. Then, every eigenvalue of $A$ is contained in the union $$\bigcup_{i=1}^n B^\CC_{R_i}(a_{ii}),$$ where $B^\CC_\delta(z)\,=\,\{\tilde z \in\CC: |z-\tilde z|\leq\delta\}$ and $R_i\,=\,\sum_{j\ne i}|a_{ij}|$ for every $i=1,...,n.$ 

\end{theorem}

\medskip

 In particular, if we apply Theorem~\ref{circlethm} to a $0-1$ matrix $A$ and its transpose, we get 
\begin{eqnarray}\label{eqcircle}
     &r(A)\,\,&\leq\,\,\,\, \max_{i} \# \{j\,:\,a_{ij}=1\}\nonumber \\
        r(A)\,=\,&r(A^{T})\,&\leq\,\,\,\,  \max_{j} \# \{i\,:\,a_{ij}=1\}.     
\end{eqnarray}

\smallskip

In the following example we will apply the estimates above to compute the metric mean dimension on a class of subshifts of compact type.

\medskip

\subsection{Delayed full shifts}
Let us consider a family of maps which are prescribed by a closed set and a collection of finite transitions which connect any pair of points in the set. We will use the information of the previous subsection with a slight modification: the only difference in the method is that we code different elements of the cover by the same symbol as long as they behave in a similar manner.

\begin{proposition}\label{delayformula}
Let $(X,d)$ be a compact metric space. Fix a positive integer $k\geq2$ and select distinct points $a_1,...,a_{k-1}\in X$ and a closed subset $F\subset X$.
 Consider $\Ga\subset X\times X$ given by $$\Ga=(F\times \{a_1\})\cup\{(a_1,a_2),(a_2,a_3),...,(a_{k-2},a_{k-1}) \}\cup(\{a_{k-1}\}\times F).$$ Then,
 $$\overline{\mathrm{mdim}}_M (\Ga_\NN,d^\NN,\sigma_\NN)=\frac{\overline{\dim}_B(F)}{k}\qquad\text{and}\qquad\underline{\mathrm{mdim}}_M (\Ga_\NN,d^\NN,\sigma_\NN)=\frac{\underline{\dim}_B(F)}{k}.$$
\end{proposition}

\begin{proof} The computations below do not depend on whether we take upper or lower limits in $\ep$, so we do not distinguish them by the notation. For a given $\ro>1$, we will denote the metric defined in \eqref{metricshift} by $d^\ro$.

Since the map $h\colon(F^\NN,d^{2^k})\to (\Ga_\NN,d^2)$ defined as $$h(x_1,x_2,...)=(x_1,a_1,a_2,...,a_{k-1},x_2,a_1,a_2,...,a_{k-1},x_3,...)$$ is an isometry over its image and $h\circ\sigma_\NN=\sigma_\NN^k \circ h$, by Proposition~\ref{conj} and \eqref{fullshift} we have
\begin{alignat*}{3}
 \dim_B(F,d)\,&=\,\mathrm{mdim}_M (F^\NN,\sigma_\NN,d^{2^k})\,&&=\,\mathrm{mdim}_M (h(F^\NN),\sigma_\NN^k,d^{2}) \\ &\leq\,\mathrm{mdim}_M (\Ga_\NN,\sigma_\NN^k,d^{2}) \,&&\leq\,k\,\mathrm{mdim}_M (\Ga_\NN,\sigma_\NN,d^{2}).     
\end{alignat*}

For the converse inequality, we may assume without loss of generality that $a_j\in F$ for every $j=1,...,k-1$. Otherwise, the set $\Ga$ would be contained in $\Ga'$ corresponding to $F'= F\cup\{a_1,...,a_{k-1}\}$, whose upper and lower box dimension are equal to the ones of $F.$

Given $\ep>0$ we consider a minimal open $\ep-$cover $U_1,...,U_N$ of $F$, that is, $N=S(F,d,\ep)$. We can assume that $a_j \in U_j$ uniquely. For each $n\geq1$, consider the following open cover  of $\Ga_\NN$ by cylinders (see \eqref{cylinderdef}) $$\alpha_{n,\ep} \,  =\, \big\{ C(\Ga_\NN\,;\,1,\cdots,n\,;\,U_{i_1},\cdots, U_{i_n})\,:\,1\leq i_1,...,i_{n}\leq N \,\big\} $$  
In order to estimate the cardinality of $\alpha_{n,\ep}$, we partition its elements in the following way: for each $j=1,...,k-1$, let
\begin{align*}
    x_n^0 \,&= \,\#\big\{C(\Ga_\NN\,;\,1,\cdots,n\,;\,U_{i_1},\cdots, U_{i_n})\,\in\,\alpha_{n,\ep}\, :\, i_n\geq k\big\} \\ x_n^j\, &= \,\#\big\{C(\Ga_\NN\,;\,1,\cdots,n\,;\,U_{i_1},\cdots, U_{i_n})\,\in\,\alpha_{n,\ep}\, :\, i_n=j\big\}.
\end{align*}
Note that

$$ 
\left\{ \begin{array}{lll}
     x_{n+1}^0\,&=\,\,\,\,x_n^{k-1}(N-k+1)& \\ x_{n+1}^1\,&=\,\,\,\,x_n^0 + ...+x_n^{k-1}& \\ x_{n+1}^j\,&=\,\,\,\,x_{n}^{j-1} +x_n^{k-1}\,\,&\,\forall\,j\in\{2,...,k-1\},
\end{array} \right.
$$
or equivalently,
\begin{equation*}
   \begin{pmatrix} x_{n+1}^0\\x_{n+1}^1\\ \vdots \\x_{n+1}^{k-1} \\ \end{pmatrix} 
    = \begin{pmatrix} 0&0&0&0& \cdots & 0 &N-k+1
    \\1&1&1&1&\cdots&1&1
    \\0&1&0&0&\cdots&0&1 
    \\0&0&1&0&\cdots&0&1
    \\0&0&0&1&\cdots&0&1
    \\\vdots&\vdots&\vdots&0&\ddots&0&1
    \\0&0&0&0&0&1&1
    \\ \end{pmatrix}
    \begin{pmatrix} x_{n}^0\\x_{n}^1\\ \vdots \\x_{n}^{k-1} \\ \end{pmatrix} 
\end{equation*}
We denote the above equality by $v_{n+1}=Av_n$, where $A=A_\ep$ and $v_1=(N-k+1,1,1,...,1)$. Then, $$\#\alpha_{n,\ep}\,=\,x_{n+1}^1 \,\leq\, |v_{n+1}| \,\leq \, ||A^n||\, |v_1|. $$ Let $\ell=\ell(\ep)\geq1$ be a positive integer satisfying $2^{-\ell}\,\diam(X,d)\leq\ep$. Then $\diam(\alpha_{n+\ell(\ep),\ep},d^2_n)\leq\ep$ and we get \begin{equation}\label{delay1}
    h_\ep(\Ga_\NN,d^2,\si_\NN) \,\leq\, \lim_{n\,\to\,+\infty} \frac{1}{n} \log \# \alpha_{n+\ell(\ep),\ep} \,=\, \lim_{n\,\to\,+\infty} \frac{1}{n} \log |x_{\ell+n+1}^1| \,\leq\, \log r(A_\ep).
\end{equation}
On the other hand, it is easy to see that the entries of $A^k$ are polynomials in $N$ of degree at most $1$. Consequently, there exists $C\geq1$ independent of $N$ satisfying $||A^k||\,\leq\, CN$. Combining this information with \eqref{delay1} we obtain $$ h_\ep(\Ga_\NN,d^2,\si_\NN) \,\leq\, \log r(A_\ep)\,\leq\, \log ||A_\ep^k||^{1/k}\,\leq\, \frac{1}{k} \Big(\log C\,+\, \log S(F,d,\ep)\Big).$$ We finish the proof by dividing by $\log(1/\ep)$ and making $\ep\to0^+$.
\end{proof}
\smallskip

\section{Proof of Theorem~\ref{gausslikethm}}\label{proofC}

In this section we fix a compact interval $I$, a compact subset $A\subset I$ with zero Lebesgue measure and the associated family of cut-out sets $(J_k)_{k\geq1}$, ordered non-increasingly in diameter. For each $k\geq1$, let $\ep_k=|J_k|$ and denote the Euclidean distance in $I$ by $d$. From now on we will assume that $\ep_k>0$ for every $k\geq1.$ If, otherwise, there is $k_0$ such that $J_{k_0}=\emptyset$, then $T_A$ has finitely many branches and so, by \eqref{entropy}, its entropy is finite; therefore its metric mean dimension is zero, trivially satisfying \eqref{eqgausslike}.

For the sets $A$ under consideration, the following dimension estimates via cut-out sets are provided by \cite[Propositions~3.6 and 3.7]{Fa2}.

\begin{lemma}\label{falcestimates}
     Let $A\subset I$ and $(\ep_k)_{k\geq1}$ be as above. Then $$\liminf_{k\,\to\,+\infty}\frac{\log k}{\log(1/\ep_k)}\,\leq\, \underline{\mathrm{dim}}_B(A,d)  \,\leq\, \overline{\mathrm{dim}}_B(A,d)\,=\,\limsup_{k\,\to\,+\infty}\frac{\log k}{\log(1/\ep_k)}.$$ Moreover, the box dimension of $A$ exists if, and only if, the outer limits coincide.
\end{lemma}

Theorem~\ref{gausslikethm} is a consequence of the previous lemma together with the following inequalities, which are the content of Propositions~\ref{lowerboundgauss} and \ref{upperboundgauss}:
\begin{eqnarray*}
    \limsup_{k\,\to\,+\infty}\frac{\log k}{\log(1/\ep_k)}&\leq\,\,\,\, \overline{\mathrm{mdim}}_M(I,d,T_A)&\leq\,\,\,\,\overline{\mathrm{dim}}_B(A,d) \\ \liminf_{k\,\to\,+\infty}\frac{\log k}{\log(1/\ep_k)}&\leq\,\,\, \, \underline{\mathrm{mdim}}_M(I,d,T_A)&\leq\,\,\,\,\underline{\mathrm{dim}}_B(A,d).
\end{eqnarray*}

\smallskip

To prove the following proposition, the differentiability assumption (C3) in the statement of the theorem is not needed.

\begin{proposition}\label{lowerboundgauss}
    Let $T_A\colon I\to I$ be a map satisfying conditions \emph{(C1)} and \emph{(C2)}. Then $$            \limsup_{k\,\to\,+\infty}\frac{\log k}{\log(1/\ep_k)}\,\leq\, \overline{\mathrm{mdim}}_M(I,d,T_A) \quad\text{and}\quad\liminf_{k\,\to\,+\infty}\frac{\log k}{\log(1/\ep_k)}\,\leq\, \underline{\mathrm{mdim}}_M(I,d,T_A)).$$
\end{proposition}
\begin{proof}
 By Definition~\ref{discontdef} of metric mean dimension, we must consider $\Ga=\graph(T_A)$ and the subset $\Ga_\NN\subset I^\NN$ endowed with the metric $d^\NN.$

    Given $k\geq1$, each bijectivity domain $J_1,J_2,...,J_{2k}$ have length at least $\ep_{2k}.$ Let $i_1,...,i_{2k}$ be a reordering of $\{1,...,2k\}$ so that $J_{i_1},...,\,J_{i_{2k}}$ are ordered from left to right. Then the sets $J_{i_1},J_{i_3},...,J_{i_{2k-1}}$ are pairwise separated by at least $\ep_{2k}.$ Since condition (C2) ensures that each cylinder composed by $J_{i_1},J_{i_3},...,J_{i_{2k-1}}$ is nonempty, for every $n\geq1$ $$S_1(\Ga_\NN,d^\NN_n ,\ep_{2k})\,\geq\,k^n.$$ This implies that $$h_{\ep_{2k+1}}(\Ga_\NN,d^\NN,\si_\NN)\,\geq\,h_{\ep_{2k}}(\Ga_\NN,d^\NN,\si_\NN)\,\geq\,\log k\qquad\forall k\geq1.$$ Thus, $$h_{\ep_{m}}(\Ga_\NN,d^\NN,\si_\NN)\,\geq\,\log \big(\frac{m}{2}-1\big)\qquad\forall m\geq1.$$ Given $\ep>0$, take $m=m(\ep)$ satisfying $\ep_{m+1}<\ep\leq\ep_m.$ Thus, $$\frac{h_\ep(\Ga_\NN,d^\NN,\si_\NN)}{\log(1/\ep)}\,\geq\,\frac{h_{\ep_m}(\Ga_\NN,d^\NN,\si_\NN)}{\log(1/\ep_{m+1})}\,\geq\, \frac{\log(m+1)}{\log(1/\ep_{m+1})}\,\frac{\log(\frac{m}{2}-1)}{\log(m+1)}.$$ We finish the proof by making $\ep\to0^+$ (and consequently $m\to+\infty$).
\end{proof}

\smallskip

\begin{proposition}\label{upperboundgauss}
     Let $T_A\colon I\to I$ be a map satisfying conditions \emph{(C1)-(C3)}. Then \begin{eqnarray*}
         \overline{\mathrm{mdim}}_M(I,d,T_A)&\leq& \overline{\mathrm{dim}}_B(A,d)\\\underline{\mathrm{mdim}}_M(I,d,T_A)&\leq& \underline{\mathrm{dim}}_B(A,d).
     \end{eqnarray*}
\end{proposition}
\begin{proof}
 Assume for simplicity that $I=[0,1]$ and fix $\ep>0$. By Lemma~\ref{specrad} and \eqref{eqcircle} we have
 \begin{align}\label{111}
      h_\ep(\Ga_\NN,d^\NN,\si_\NN)\,\, \leq\, \, \log r(\Ga^\ep)  \,\,\leq\,\,\log \max_{1\leq j\leq\lceil 1/\ep\rceil} \#\{ i:\Ga^\ep_{ij}=1\},
 \end{align}
where $\Ga^\ep$ is defined as in \eqref{gammaep}. For each $j=1,...,\lceil 1/\ep\rceil$ fix any point $x_j\in\big(\ep(j-1),\ep j\big)\cap[0,1]$ and observe that $T_A^{-1}(x_j)\,=\,\{a_k(j)\}_{k}$ is a sequence that intersects each $J_k$ exactly once, by conditions (C1) and (C2).

Let $M$ be a positive integer such that $M^{-1}<\eta$, where $\eta$ is as in (C3). The graph of $T_A|_{J_k}$ can never intersect more than $M+2$ boxes in the $\ep-$grid per line. Otherwise, by the Mean Value Theorem there would exist some point $z$ in the interior of $J_k$ such that $|T'(z)|<M^{-1}.$ Hence, for every $j=1,...,\lceil 1/\ep\rceil$ we have
\begin{align}\label{222}
    \#\{ i:\Ga^\ep_{ij}=1\}\,&\nonumber=\,\#\{ i: \big( [(i-1)\ep,i\ep] \times[(j-1)\ep,j\ep] \big) \cap \Ga \ne\emptyset\}  \\&\nonumber\leq\, 2(M+2) \,\#\{i: [\ep(i-1),\ep i]\cap T_A^{-1}(x_j) \ne\emptyset\} \\&\nonumber \leq\, 4(M+2)\, S(T_A^{-1}(x_j),d,\ep) \\& \leq\, 8(M+2)\, S(A,d,\ep),
\end{align}
\smallskip
where the second inequality is due to the fact that every open interval of diameter $\ep$ can intersect at most $2$ intervals of the form $[(i-1)\ep,i\ep]$, and the last inequality is a consequence of the next lemma.

\begin{lemma}\label{inbetween}
    Let $(a_k)_{k\geq1}$ be a sequence of positive real numbers such that $a_k\in J_k$, for every $k\geq1$. Then for every $\ep>0$, $$S(\{a_k\colon k\geq1\}, d,\ep)\,\leq\,2\,S(A,d,\ep).$$
\end{lemma}
\begin{proof}
    Given a minimal collection of intervals of diameter at most $\ep$ that covers $A$, there can be at most one element of $\{a_k\}_k$ in between each pair of consecutive intervals not yet covered. Adding one extra interval for each of these (at most $S(A, d,\ep)$) points, we get an $\ep-$cover of $\{a_k\}_k$ with cardinality $2S(A, d,\ep)$.
\end{proof}

Combining \eqref{111} and \eqref{222}, we get
$$h_\ep(\Ga_\NN,d^\NN,\si_\NN)\,\leq\,\log \big(8(M+2)\big)\,+\, \log S(A,d,\ep), $$
and we end the proof of Proposition~\ref{upperboundgauss} after dividing by $\log(1/\ep)$ and making $\ep\,\to\,0^+.$ \end{proof}

\subsection{Large derivative assumption}\label{C3section}  Condition (C3) was used in the previous reasoning to obtain a uniform upper bound for the number of entries equal to $1$ in every row of $\Ga^\ep$ in terms of the $\ep-$covering number of $A$. However, it is not strictly necessary as we will illustrate.

\begin{example}
 Let us compute the metric mean dimension, with respect to the Euclidean distance $d$, of the following (discontinuous) interval map 
$$
\begin{array}{rccc}
T \colon & [-1,1] & \longrightarrow & [-1,1] \\
& x\ne0  & \mapsto & \sin \frac{1}{|x|} \\ & 0 & \mapsto & 0
\end{array}.
$$ We start by observing that the map $T$ does not satisfy either of the conditions (C2) (since $T$ restricted to the left and rightmost intervals is not full branch) and (C3) (since there are infinitely many differentiable maxima and minima). In order to show that (C3) is not strictly necessary for Theorem~\ref{gausslikethm} to be valid, we must consider a map satisfying (C2). After computing the metric mean dimension of $T$, we will consider a slightly modified map satisfying (C2) and such that all the estimates still hold.

The map $T$ can be seen as $T_{A}$ satisfying condition (C1) and, up to the left and rightmost intervals, also (C2), where $A\,=\,\big\{ \frac{2}{(2k+1)\pi},\,\frac{-2}{(2k+1)\pi} \,:\,k\geq0 \big\}\,\cup\,\{0\}.$ Let us prove that
\begin{equation}\label{sinmdim}
\mathrm{mdim}_M([-1,1],d,T)\,=\,\frac{1}{2}\,=\,\dim_B(A).
\end{equation}

Firstly, an argument similar to the one in the proof of Proposition~\ref{lowerboundgauss} yields $$\underline{\mathrm{mdim}}_M([-1,1],d,T)\,\geq\,1/2.$$

For the converse inequality, we cannot directly apply Proposition~\ref{upperboundgauss} since the map $T$ does not satisfy (C3). The main idea is still to make use of Lemma~\ref{specrad} and \eqref{eqcircle}. However, we will estimate the number of $1$'s in a column by making use of the Second Order Mean Value Theorem.\footnote{ If $f\colon[a,b]\to\RR$ is twice differentiable, then $f(b)\,+\,f(a)\,-\,2\,f(\frac{a+b}{2}) \,=\, \frac{1}{4}\,(b-a)^2\,f''(c)$ for some $c\in(a,b).$}

For each $n\geq1$, fix the scale $\ep_n\,=\,\Big\lceil \frac{3\pi(4n+1)(4n+5)}{8} \Big\rceil ^{-1}$ and consider an $\ep_n-$grid. To estimate the spectral radius
of $\Ga^{\ep_n} $ - defined in \eqref{gammaep} - we restrict our attention to the first quadrant, as the remaining ones can be dealt in an analogous way. 

Consider the positive maximum points $x_k=2/(4k+1)$, $k\geq0$. Given $y\in(0,1]$ and $L\in(0,1)$ such that $T(y)=\sin(1/y)\geq L$, which implies $|\cos(1/y)|\leq\sqrt{1-L^2}$, we have two possibilities:
\begin{itemize}
    \item[(\emph{i})]$y\in(0,x_0]$

    \smallskip
    
   \noindent Let $k\geq1$ be such that $x_{k+1}\leq y\leq x_{k-1}$. Then,
\begin{align*}
     |f''(y)|\, &=\, \big|y^{-4}\,\sin(1/y)\,+\,2\,y^{-3}\,\cos(1/y)\big| \, \\&\geq\, y^{-4}\,L\,-\,2\,y^{-3}\,\sqrt{1-L^2} \\ & \geq\, x_{k-1}^{-4}\,L\,-\,2\,x_{k+1}^{-3}\,\sqrt{1-L^2} \\& =\, k^4\, \frac{\pi^3}{8} \bigg[ \frac{\pi L}{2}\,(4-\frac{3}{k})^4\,-\,2\,\sqrt{1-L^2}\,\Big(\frac{4}{k^{1/3}}+\frac{5}{k^{4/3}}\Big)^3 \bigg] \\ &\geq\,k^4\, \frac{\pi^3}{8} \bigg[ \frac{\pi L}{2}\,-\,2\,\sqrt{1-L^2}\,9^3 \bigg] 
\end{align*} Thus, $|f''(y)|\,\geq\, k^4$ if $L$ is close enough to $1$.
\smallskip
    \item [(\emph{ii})]$y\in(x_0,1]$    
    \begin{align*}
        |f''(y)|\, &=\, \big|y^{-4}\,\sin(1/y)\,+\,2\,y^{-3}\,\cos(1/y)\big| \, \\&\geq\, y^{-4}\,L\,-\,2\,y^{-3}\,\sqrt{1-L^2} \\ & \geq\,L\,-\,2\,x_{1}^{-3}\,\sqrt{1-L^2} \\& =\,L\,-\,2\Big(\frac{5\pi}{2}\Big)^{3}\,\sqrt{1-L^2}.
    \end{align*} Thus, $|f''(y)|\,\geq\, 1/2$ if $L$ is close enough to $1$.
\end{itemize}

\smallskip

From now on, we fix $L\in(0,1)$ close enough to $1$ so that the two previous estimates hold. Take $n\geq1$ big enough so that $1-\ep_n \gg L$. Let us bound the spectral radius of $\Ga^{\ep_n}$ from above by estimating the number of $1$'s. To do so, we recall from \eqref{eqcircle} that $$r(\Ga^{\ep_n})\,\leq\,  \max_{j} \# \big\{i\,:\,(\Ga^{\ep_n})_{ij}=1\big\}$$ and separate $j$ in $3$ classes. First, note that $\ep_n\,<\,\frac{x_n-x_{n+1}}{3}$ for every $n$. This implies that every interval $[(i-1)\ep_n,i\ep_n]$, $i=1,...,\,\lceil \frac{x_n}{\ep_n} \rceil$, contains at least two $x_k$'s. Therefore, \begin{equation}\label{ones}
    (\Ga^{\ep_n})_{i,j}\,=\,1\qquad \forall i=1,...,\,\bigg\lceil \frac{x_n}{\ep_n} \bigg\rceil,\,\,\forall j=1,...,\frac{1}{\ep_n}.
\end{equation}

\noindent\textbf{Case 1:} $j=1/\ep_n$. 

\smallskip

\noindent For each $k\geq0$, consider the nearest points $a_k<x_k<b_k$ such that $T(a_k)=T(b_k)=1-\ep_n$. Then $$\{x\in(0,1)\,:\,T(x)\geq1-\ep_n\}\,=\,\bigcup_{k\geq0}[a_k,b_k].$$ By the Second Order Mean Value Theorem, for every $k\geq0$ there exists $c_k\in(a_k,b_k)$ such that $|b_k-a_k|^2\,\leq\,\frac{8\ep_n}{|f''(c_k)|}$. Moreover, by the choice of $L$ we have $|b_0-a_0|\,\leq\,4\sqrt{\ep_n}$ and $|b_k-a_k|\,\leq\,\frac{\sqrt{8}\,\sqrt{\ep_n}}{k^2}$, for every $k\geq1$. Combining this information with \eqref{ones}, we obtain 
\begin{align}\label{topline}
    \# \big\{i\,:\,(\Ga^{\ep_n})_{i,\frac{1}{\ep_n}}=1\big\}\,\,&\leq\,\,\bigg\lceil \frac{x_n}{\ep_n} \bigg\rceil\,+\,\sum_{k=0}^n \Big( \bigg\lfloor \frac{|b_k-a_k|}{\ep_n} \bigg\rfloor\,+\,2 \Big)\nonumber \\ &\leq\,\,\bigg\lceil \frac{x_n}{\ep_n} \bigg\rceil\,+\,2(n+1)\,+\,\frac{1}{\sqrt{\ep_n}}\Big(4\,+\, \sqrt{8}\,\sum_{k=1}^\infty k^{-2} \Big)\nonumber  \\ &\leq\,\, C\,n,
\end{align} where $C>0$ is a constant independent of $n$.

\medskip

\noindent\textbf{Case 2:} $j\,=\,\lceil L/\ep_n\rceil\,+\,1\,,...,\,1/\ep_n\,-\,1.$

\smallskip

\noindent For each $k\geq0$, consider the nearest points $a^1_k<b_k^1<x_k<a_k^2<b_k^2$ such that $T(a^1_k)=T(b_k^2)=(j-1)\ep_n$ and $T(a^2_k)=T(b_k^1)=j\,\ep_n$. Then $$\big\{x\in(0,1)\,:\,(j-1)\ep_n\,\leq\,T(x)\,\leq\,i\,\ep_n\big\}\,=\,\bigcup_{k\geq0}[a_k,b_k].$$ By the Second Order Mean Value Theorem, for every $k\geq0$ and $s=1,2$ there exists $c^s_k\in(a^s_k,b^s_k)$ such that $|b^s_k-a_k^s|^2\,\leq\,\frac{8\ep_n}{|f''(c^s_k)|}$. Again, by the choice of $L$ we have $|b^s_0-a^s_0|\,\leq\,4\sqrt{\ep_n}$ and $|b^s_k-a^s_k|\,\leq\,\frac{\sqrt{8}\,\sqrt{\ep_n}}{k^2}$, for every $k\geq1$. Joining this information with \eqref{ones}, we obtain 
\begin{align}\label{upperlines}
    \# \big\{i\,:\,(\Ga^{\ep_n})_{i,j}=1\big\}\,\,&\leq\,\,\bigg\lceil \frac{x_n}{\ep_n} \bigg\rceil\,+\,\sum_{s=1}^2\sum_{k=0}^n \Big( \bigg\lfloor \frac{|b^s_k-a^s_k|}{\ep_n} \bigg\rfloor\,+\,2 \Big)\nonumber \\ &\leq\,\,\bigg\lceil \frac{x_n}{\ep_n} \bigg\rceil\,+\,4(n+1)\,+\,\frac{2}{\sqrt{\ep_n}}\Big(4\,+\, \sqrt{8}\,\sum_{k=1}^\infty k^{-2} \Big)\nonumber  \\ &\leq\,\, C\,n,
\end{align}for every $j=\lceil L/\ep_n\rceil\,+\,1,...,\frac{1}{\ep_n}\,-\,1,$ where $C>0$ is a constant independent of $n$.

\medskip

\noindent\textbf{Case 3:} $j\,=\,1,...,\,\lceil L/\ep_n\rceil.$

\smallskip

\noindent By taking $n$ large enough, we may assume that $L\,+\,\ep_n\,<\,\frac{1+L}{2}$. Thus, $T(x)=\sin(1/x)\leq \frac{1+L}{2} $, for every $x\in(0,1)$ such that $T(x)\in[(j-1)\ep_n,j\,\ep_n]$ for some $j\,=\,1,...,\,\lceil L/\ep_n\rceil$. Hence, there exists some $M=M(L)\in\NN$ such that, for every such $x$, we have $$|T'(x)|\,=\,\frac{|\cos(1/x)|}{x^2}\,\geq\,\sqrt{1-\Big(\frac{1+L}{2}\Big)^2} \,\geq\, M^{-1}.$$ By the proof of Proposition~\ref{upperboundgauss} we obtain 
\begin{equation}\label{lowerlines}
     \# \big\{i\,:\,(\Ga^{\ep_n})_{i,j}=1\big\}\,\,\leq\,\, 8(M+2)\,n,
\end{equation} for every $j\,=\,1,...,\,\lceil L/\ep_n\rceil$.

\medskip

Bringing together \eqref{topline}, \eqref{upperlines} and \eqref{lowerlines}, and proceeding analogously in the other quadrants, we conclude that there exists a constant $C>0$, which depends only on $L$, such that $$r(\Ga^{\ep_n})\,\leq\, \max_{j} \# \{i\,:\,(\Ga^{\ep_n})_{ij}=1\}\,\leq\, C\,n,$$ for every large enough $n\in\NN$.

We are ready to complete the computation of the upper bound of the metric mean dimension of $T$. Given $\ep>0$, take $n=n(\ep)$ such that $\ep_{n}\leq\ep<\ep_{n-1}$. Then $$  \frac{h_{\ep}(\Ga_\NN,d^\NN,\si_\NN)}{\log(1/\ep)}\,\,\leq\,\,  \frac{h_{\ep_{n}}(\Ga_\NN,d^\NN,\si_\NN)}{\log(1/\ep_{n-1})} \,\,\leq\,\, \frac{r(\Ga^{\ep_n})}{\log(1/\ep_{n-1})}\,\,\leq\,\,\frac{\log(C\,n)}{\log(1/\ep_{n-1})}, $$ and so $$\overline{\mathrm{mdim}}_M([-1,1],d,T)\,=\,\overline{\mathrm{mdim}}_M(\Ga_\NN,d^\NN,\si_\NN)\,\leq\,\limsup_{n\,\to\,+\infty}\frac{\log(C\,n)}{\log(1/\ep_{n-1})}\,=\,\frac{1}{2}.$$
\end{example}

\medskip

We finish this section by observing that all arguments used to compute the metric mean dimension of the map $T$ remain valid if we consider a smooth perturbation $\tilde{T}$ of $T$ in small neighbourhoods $(1-\delta,1]$ and $[-1,-1+\delta)$, so that $\tilde{T}(1)=\tilde{T}(-1)=-1$ and $|\tilde{T}'(x)|>M^{-1}$ for every $x\,\in\,[-1,-1+\delta)\,\cup\,(1-\delta,1]$. In this case, we obtain a map satisfying (C1) and (C2) such that the equality \eqref{eqgausslike} in Theorem~\ref{gausslikethm} still holds, even though the assumption (C3) does not.

\subsection{A remark on the definition of metric mean dimension of discontinuous maps}\label{rmkdiscontdef}    Given a discontinuous map $T$, its associated transition set $\Ga=\graph(T)$ is not closed. In order to associate a subshift of compact type to $T$, we could have followed two possible paths: one option would be to take \emph{the closure of the transition set} and afterwards consider its associated sequence space $(\overline{\Ga})_\NN$; another way (which is the one we chose) is to consider \emph{the closure of the sequence space} $\overline{(\Ga_\NN)}$. At first glance, the former might appear more satisfactory since, by Theorem~\ref{n=z}, we would obtain a notion of metric mean dimension that could be expressed both by unilateral and bilateral sequences. Let us exemplify why it would not yield a \emph{meaningful} notion of metric mean dimension.

Consider the set $A=\{e^{-n}:n\geq0\}\cup\{0\}$ and let $T_A\colon [0,1]\to[0,1]$ be any map satisfying conditions (C1)-(C3). The arguments in the proof of Theorem~\ref{gausslikethm} do not depend whether we define $\mathrm{mdim}_M([0,1],d,T_A)$ in terms of $\overline{(\Ga(T_A)_\NN)}$ or $\big(\overline{\Ga(T_A)}\big)_\NN$. This indicates that \begin{equation}\label{meaningfulvalue}
\mathrm{mdim}_M([0,1],d,T_A)\,=\,\dim_B(A,d)\,=\,0
\end{equation}

\noindent is the right value for it. Yet, if we consider an interval map $T$ such that $T|_{[0,e^{-1})}=(T_A)|_{[0,e^{-1})}$ and $T|_{[e^{-1},1]}\equiv0$, then any meaningful notion of metric mean dimension compatible with \eqref{meaningfulvalue} should also be zero. However, since we have $$\overline{\Ga(T)}\,\supset\, \big([e^{-1},1]\times\{0\}\big)\,\cup\,\big(\{0\}\times[e^{-1},1]\big),$$ then Proposition~\ref{delayformula} implies that $$\mathrm{mdim}_M\big( (\overline{\Ga(T)})_\NN, d^\NN,\si_\NN \big)\,\geq\,\frac{1}{2}.$$
This shows that taking the closure of the transition set prior to generating the sequences may create too many new orbits that are not present in the original dynamics.

\subsection*{Acknowledgments} This work commenced during the author's Master's degree program at Universidade Federal do Rio Grande do Sul, Brazil, under the supervision of Professor Alexandre Baraviera, to whom he is deeply grateful. The author also thanks Professor Paulo Varandas for the insightful conversations and Professor Maria Carvalho for the meticulous reading and numerous suggestions that significantly improved this manuscript. The author has been awarded a PhD grant by FCT- Funda\c c\~ao para a Ci\^encia e a Tecnologia, with reference UI/BD/152212/2021.

\end{document}